\documentclass[12pt,oneside,reqno]{amsart} 
\usepackage{etex}
\usepackage[T2A]{fontenc}
\usepackage[cp1251]{inputenc}
\usepackage{pifont}
\usepackage[dvips]{graphicx}
\usepackage{amsmath,latexsym,amsthm,cmap,indentfirst,setspace,ccaption,amssymb,color,amscd,mathrsfs,hyperref}
\usepackage{array}

\usepackage{amssymb,lipsum}
\usepackage{float,xypic}
\usepackage{cmap,longtable}

\usepackage{geometry}
\usepackage{caption, subcaption}

\usepackage{verbatim}
\hypersetup{
    bookmarks=true,         
    unicode=true,          
    pdftoolbar=true,        
    pdfmenubar=true,        
    pdffitwindow=false,     
    pdfnewwindow=true,      
    colorlinks=true,       
    linkcolor=blue,          
    citecolor=blue,        
    filecolor=magenta,      
    urlcolor=cyan           
}

\DeclareMathOperator{\tr}{tr}
\DeclareMathOperator{\diag}{diag}
\DeclareMathOperator{\Spec}{Spec}
\DeclareMathOperator{\rk}{rk}
\DeclareMathOperator{\Pic}{Pic}

\DeclareMathOperator{\Fix}{Fix}
\DeclareMathOperator{\Aut}{Aut}
\DeclareMathOperator{\Bl}{Bl}
\DeclareMathOperator{\Sp}{Sp}
\DeclareMathOperator{\pr}{pr}

\DeclareMathOperator{\PGL}{PGL}
\DeclareMathOperator{\GL}{GL}
\DeclareMathOperator{\SL}{SL}
\DeclareMathOperator{\Ort}{O}
\DeclareMathOperator{\PSL}{PSL}
\DeclareMathOperator{\SO}{SO}
\DeclareMathOperator{\Eu}{Eu}
\DeclareMathOperator{\E}{E}
\DeclareMathOperator{\D}{D}
\DeclareMathOperator{\A}{A}
\DeclareMathOperator{\stab}{stab}

\DeclareMathOperator{\Exc}{Exc}
\DeclareMathOperator{\Bs}{Bs}

\DeclareMathOperator{\ord}{ord}
\DeclareMathOperator{\Gal}{Gal}

\theoremstyle{plain}
\newtheorem{thm}{Theorem}[section]
\newtheorem{lem}[thm]{Lemma} 
\newtheorem{sled}[thm]{Corollary} 
\newtheorem{prop}[thm]{Proposition} 

\theoremstyle{definition}
\newtheorem{mydef}[thm]{Definition} 
\newtheorem{rem}[thm]{Remark}
\newtheorem{ex}[thm]{Example}

\makeatletter
\g@addto@macro{\endabstract}{\@setabstract}
\newcommand{\authorfootnotes}{\renewcommand\thefootnote{\@fnsymbol\c@footnote}}%
\makeatother	

\begin{document}
	
\begin{center}
	\LARGE 
	Subgroups of odd order in the real plane Cremona group \par \bigskip
	
	\normalsize
	\authorfootnotes
	\large Egor Yasinsky\footnote{yasinskyegor@gmail.com} \par \bigskip
	\normalsize
	Steklov Mathematical Institute of Russian Academy of Sciences \\8 Gubkina st., Moscow, Russia, 119991\par \bigskip

\end{center}

\newcommand{\CC}{\mathbb C} 
\newcommand{\SSS}{\mathbb S} 
\newcommand{\FF}{\mathbb F} 
\newcommand{\TT}{\mathbb T}
\newcommand{\NN}{\mathbb N}
\newcommand{\kk}{\Bbbk}
\newcommand{\CRR}{\rm Cr_2(\mathbb{R})}
\newcommand{\CRC}{\rm Cr_2(\mathbb{C})}
\newcommand{\CRkn}{\rm Cr_n(\Bbbk)}
\newcommand{\CRk}{\rm Cr_2(\kk)}
\newcommand{\PR}{\mathbb{P}^2_\mathbb{R}}
\newcommand{\PRone}{\mathbb{P}^1_\mathbb{R}}
\newcommand{\PCone}{\mathbb{P}^1_\mathbb{C}}
\newcommand{\PC}{\mathbb{P}^2_\mathbb{C}}
\newcommand{\Pk}{\mathbb{P}^2_{\kk}}
\newcommand{\RR}{\mathbb R}
\newcommand{\PP}{\mathbb P}
\newcommand{\ZZ}{\mathbb Z}
\newcommand{\WW}{\mathcal W}
\newcommand{\DD}{\mathcal D}
\newcommand{\XC}{X_{\mathbb C}} 
\newcommand{\XR}{X_{\mathbb R}}    
\newcommand{\LL}{\mathcal L} 
\newcommand{\LLL}{\mathscr L}
\newcommand{\MM}{\mathcal M} 

\newcommand*\conj[1]{\overline{#1}}

\begin{abstract}
In this paper we describe conjugacy classes of finite subgroups of odd order in the group of birational automorphisms of the real projective plane.
\end{abstract}

\section{Introduction}\label{sec:intro}
Consider a projective space $\mathbb{P}^n_{\kk}$ over an arbitrary field $\kk$. Recall that {\it the Cremona group} $\CRkn$ is the group of its birational automorphisms. From algebraic point of view the Cremona group over $\kk$ is the group of $\kk$-automorphisms of the field $\kk(x_1,\ldots,x_n)$ of rational
functions in $n$ independent variables. 

The classification of finite subgroups in Cremona groups is a classical problem which goes back to E. Bertini's work on involutions in $\CRC$. He discovered three types of conjugacy classes, which are now known as de Jonqui\`{e}res, Geiser and Bertini involutions. However, Bertini's classification was incomplete and his proofs were not rigorous. The next step was made in 1895 by S. Kantor and A. Wiman who gave a description of finite subgroups in $\CRC$. Their list was quite comprehensive, but not precise in several respects.

The modern approach started with the works of Yu. I. Manin and V. A. Iskovskikh who discovered the deep connection between conjugacy classes of finite subgroups in the plane Cremona group and classification of $G$-minimal rational surfaces $(S,G)$ and $G$-equivariant birational maps between them. This approach was taken by L. Bayle and A. Beauville in their work on involutions \cite{bb}. The classification was generalised by T. de Fernex for subgroups of prime order \cite{df}. Finite abelian subgroups in $\CRC$ were classified by J. Blanc in \cite{blancabelian}. Finally, the most precise description of conjugacy classes of arbitrary finite subgroups in $\CRC$ was given by I. V. Dolgachev and V. A. Iskovskikh in \cite{di}. 

Much less is known in the case when the ground field $\kk$ is not algebraically closed or $n\geqslant 3$. Some results about the existence of birational automorphisms of prime order in $\CRk$ for any perfect field $\kk$ were obtained by Dolgachev and Iskovskikh in \cite{di-perf}. Similar questions, including a Minkowski-style bound for the orders of the finite subgroups
in $\CRk$, are discussed in J.-P. Serre's works \cite{serre2}, \cite{serre1}. The generators for various subgroups of $\CRR$ were studied by J. Blanc and F. Mangolte in \cite{blanc}. As for the case $n\geqslant 3$, there are only partial results on classification of finite subgroups in ${\rm Cr}_3(\CC)$ (see, for example, \cite{pro1}, \cite{pro2}).

In this paper we work in the category of schemes defined over $\RR$ together with regular morphisms of schemes. In other words, a regular morphism for us is a rational map defined at all complex points. The group of automorphisms of a scheme $X$ in such a category is denoted by $\Aut(X)$.  One can also consider the category with the same objects and morphisms defined as follows: we say that there is a morphism $f:X\dashrightarrow Y$ if $f$ is a rational map defined at all real points of $X$. Automorphisms in such a category are called {\it birational diffeomorphisms} and the corresponding group is denoted by $\Aut(X(\RR))$. Clearly, $\Aut(X)\subset\Aut(X(\RR))$. In recent years, birational diffeomorphisms of real rational projective surfaces have been studied intensively (see, for example, \cite{huis-mang}, \cite{kol-mang}). In particular, prime order birational diffeomorphisms of the sphere, i.e. elements of the group $\Aut(S(\RR))$, where $S=\{[w:x:y:z]\in\PP_\RR^3: w^2=x^2+y^2+z^2\}$, were studied in \cite{rob}. 

In this work we classify all subgroups of odd order in the real plane Cremona group. Our main results are the following two theorems.
\begin{thm}\label{conicbundle}
Any finite subgroup of odd order in $\CRR$ is conjugate to a subgroup of the automorphism group of some real del Pezzo surface $X$. More precisely, one of the following holds:
\begin{enumerate}
\item $\rk\Pic(X)^G=1$, and $X$ is $\RR$-rational;
\item $\rk\Pic(X)^G=2$, $X\cong \PRone\times\PRone$ and $G$ can be written as a direct product of at most two cyclic groups.
\end{enumerate}
\end{thm}
The next theorem gives the details about finite groups arising in the case (1) of Theorem \ref{conicbundle}. It also provides an answer to the question which of those groups are {\it linearizable}, i.e. conjugate in $\CRR$ to subgroups of $\Aut(\PP_\RR^2)=\PGL_3(\RR)$ (see subsection \ref{subsec:regularization} for precise definitions).
\begin{thm}\label{delpezzo}
Let $X$ be a real $\RR$-rational del Pezzo surface, and $G\subset\Aut(X)$ be a group of odd order, such that $\rk\Pic(X)^G=1$. Then one of the following cases holds:
\begin{itemize}
\item $K_X^2=9$, $G$ is a cyclic subgroup of $\PGL_3(\RR)$;
\item $K_X^2=8$, $G$ is cyclic and linearizable;
\item $K_X^2=6$, $G\cong(\ZZ/n\ZZ\times\ZZ/m\ZZ)\rtimes(\ZZ/3\ZZ)$ for some odd integers $n,m\geq 1$; this group is linearizable if and only if $n=m=1$;
\item $K_X^2=5$, $G\cong\ZZ/5\ZZ$ and is linearizable.
\end{itemize}
Moreover, all the possibilities listed above actually occur.
\end{thm}
\begin{rem}\label{rem: counterexample}
	It may be interesting to notice that Theorem \ref{delpezzo} with a slight modification is valid if we replace the $\RR$-rationality assumption by a weaker one, namely $X(\RR)\ne\varnothing$. The only new cases obtained are del Pezzo surfaces of degree 2 and 3 with non-connected real loci (hence these are not $\RR$-rational), and the group $\ZZ/3\ZZ$ acting minimally on them (see Example \ref{ex: counterexample} and Remark \ref{rem: dP2 non rational}).
\end{rem}

This paper is organised as follows. Section \ref{sec:prelim} recalls notation and background results from the theory of rational surfaces and equivariant minimal model program. In Section \ref{sec:cb} we prove Theorem ~\ref{conicbundle}. In Section \ref{sec:bigdp} we prove Theorem \ref{delpezzo} for del Pezzo surfaces of high degree. In particular, we show that there is an infinite series of non-linearizable subgroups of odd order in $\CRR$ and give an example of two non-conjugate embeddings of $(\ZZ/3\ZZ)^2$ into $\CRR$. In Section ~\ref{sec: smalldp} we prove Theorem ~\ref{delpezzo} for del Pezzo surfaces of low degree. Our methods here are less geometric (namely, we use the Weyl groups). Finally, for the reader's convenience, some information about conjugacy classes in the Weyl groups is included in Appendix ~\ref{appendix}.

\subsection*{Acknowledgments} This work was performed in Steklov Mathematical Institute and supported by the Russian Science Foundation under grant 14-50-00005. The author would like to thank his advisor Yuri Prokhorov for numerous useful discussions. The author is also grateful to J\'er\'emy Blanc, Fr\'{e}d\'{e}ric Mangolte, Constantin Shramov, Andrey Trepalin and the anonymous referee for their useful remarks.

\section{Preliminaries}\label{sec:prelim}

Throughout the paper $X$ denotes geometrically smooth projective real algebraic surface, and $\XC$ denotes its complexification:
\[
\XC=X\times_{\Spec\RR}\Spec\CC.
\]
Note that there is a natural Galois group $\Gamma=\Gal(\mathbb{C}/\mathbb{R})=\langle\sigma\rangle_2$ action on $\XC$ (here and later $\sigma$ is an antiholomorphic involution on $\XC$, while $\langle a\rangle_n$ denotes a cyclic group of order $n$ generated by $a$). As usual, $X(\CC)$ denotes the set of complex points of $X$, and $X(\RR)=X(\CC)^\sigma$ is its real part (with the Euclidean topology). Consider the canonical projection \[\pr: \XC\to X.\] Let $p\in X$ be a closed point. Then either $\pr^{-1}(p)=p$ or $\pr^{-1}(p)=\{p,\sigma(p)\}$. {\it An exceptional curve} (or {\it $(-1)$-curve}) on a complex surface $S$ is a curve $L$ such that  $L\cong\PCone$ and $L^2=-1$. A curve $E$ on real surface $X$ is said to be exceptional if:
\begin{description}
\item[(i)] either $\pr^{-1}(E)=L$ is exceptional on $\XC$ and $L=\sigma(L)$;
\item[(ii)] or $\pr^{-1}(E)=L+\sigma(L)$, $L$ is exceptional on $\XC$ and $L\cap \sigma(L)=\varnothing$.
\end{description}
Recall that a surface $X$ is said to be $\RR$-{\it minimal} if any birational $\RR$-morphism $X\to Y$ to smooth projective real surface $Y$ is an isomorphism. As in the complex case, one can show that any birational morphism $X\to Y$ is a composition of blowdowns, i.e. there is a sequence of contractions of exceptional curves (in the sense of the previous definition). It follows that a surface is $\RR$-minimal if and only if it has no exceptional curves \cite[Chapter III, Theorem 21.8]{cubicforms}.

\subsection{Rational $G$-surfaces}

In the following definitions the ground field $\kk$ is an arbitrary perfect field.

\begin{mydef}
A {\it geometrically rational surface}\footnote{Note that many authors use the word ``rational'' to mean ``geometrically rational''.} $X$ is a smooth projective surface over $\kk$ such that $X_{\overline{\kk}}=X\times_{\Spec\kk}\Spec\overline{\kk}$ is birationally isomorphic to $\mathbb{P}^2_{\overline{\kk}}$. Geometrically rational surface $X$ is called {\it $\kk$-rational} if it is $\kk$-birational to $\Pk$.
\end{mydef}
\begin{mydef}
Let $G$ be a finite group. A {\it $G$-surface} is a triple $(X,G,\iota)$, where $X$ is a surface over $\kk$ and $\iota: G\hookrightarrow \Aut_{\kk}(X)$ is a faithful $G$-action. A {\it morphism} of $G$-surfaces $(X_1,G,\iota_1)\to (X_2,G,\iota_2)$ (or $G$-{\it morphism}) is a morphism $f: X_1\to X_2$ such that for each $g\in G$ the following diagram commutes:
\[\xymatrix@C+1pc{
	X_1\ar[r]^{f}\ar[d]_{\iota_1(g)}  & X_2\ar[d]^{\iota_2(g)}\\
	X_1\ar[r]^{f} & X_2  
}\]
Rational maps and birational maps of $G$-surfaces are defined in a similar way. We will often omit $\iota$ from the notation and refer to the pair $(X,G)$ or simply $X$, if no confusion arises.
\end{mydef}

\begin{mydef}
A $G$-surface $(X,G)$ is called {\it minimal} (we also say that $X$ is {\it $G$-minimal}) if any birational $G$-morphism $X\rightarrow X'$ of $G$-surfaces is an isomorphism. 
\end{mydef}
\begin{rem}
If $G=\{\rm id\}$ then $G$-minimal surface is just a $\kk$-minimal surface in the sense of the theory of minimal models.
\end{rem}
\begin{mydef}
Let $f: X\rightarrow B$ be a $G$-morphism of $G$-surface $(X,G)$, where $B$ is a curve. This morphims is said to be {\it relatively $G$-minimal} if for any decomposition \[f: X\overset{g}{\longrightarrow}X'\overset{h}{\longrightarrow}B,\] where $h$ is a $G$-morphism and $g$ is a birational $G$-morphism, $g$ is in fact an isomorphism.
\end{mydef}

\begin{mydef}
We say that a smooth $G$-surface $(X,G)$ admits a {\it conic bundle structure}, if there is a $G$-morphism $\pi: X\rightarrow C$, where $C$ is a smooth curve and each scheme fibre is isomorphic to a reduced conic in $\Pk$.
\end{mydef}
\begin{rem}
Let $\pi: X\to C$ be a geometrically rational conic bundle over a field $\kk$. If $c$ denotes the number of its singular fibres over $\overline{\kk}$, then by Noether's formula we have $K_X^2=8-c$.
\end{rem}
\begin{mydef}
A {\it del Pezzo surface} is a smooth projective surface $X$ with ample anticanonical divisor class $-K_X$. The {\it degree} $d$ of a del Pezzo surface $X$ is the self-intersection number $K_X^2$.
\end{mydef}
\begin{rem}
It is well known that a del Pezzo surface over an algebraically closed field $\conj{\kk}$ is isomorphic either to $\mathbb{P}^1_{\conj{\kk}} \times \mathbb{P}^1_{\conj{\kk}}$ or $\mathbb{P}^2_{\conj{\kk}}$ blown up in $9-d$ points in general position \cite[Chapter IV, Theorem 24.4]{cubicforms}.
\end{rem}
\begin{mydef}
Let $n\in\ZZ_{\geq 0}$. The $n$-th {\it Hirzebruch surface} (or {\it rational ruled surface}) $\FF_n$ is the projectivisation of a vector bundle $\mathscr{E}\cong\mathscr{O}_{\mathbb{P}^1_\kk}\oplus\mathscr{O}_{\mathbb{P}^1_\kk}(-n)$.
\end{mydef}

\subsection{Regularization of finite group action}\label{subsec:regularization}
Let $(X,G)$ be a rational $G$-surface. A birational map $\psi: X\dasharrow\Pk$ yields an injective homomorphism \[
i_\psi: G\to\CRk,\ \ g\mapsto \psi\circ g\circ \psi^{-1}.
\]
We say that $G$ is {\it linearizable} if there is a birational map $\psi: X\dasharrow\Pk$ such that $i_{\psi}(G)\subset\PGL_3(\kk)$. If $(X',G)$ is another rational $G$-surface with birational map $\psi': X'\dasharrow\Pk$, then the subgroups $i_\psi(G)$ and $i_{\psi'}(G)$ are conjugate if and only if $G$-surfaces $(X,G)$ and $(X',G)$ are birationally isomorphic. In other words, a birational isomorphism class of $G$-surfaces defines a conjugacy class of subgroups of $\CRk$ isomorphic to $G$.

It can be shown that any conjugacy class is obtained in this way. In fact, the modern approach to classification of finite subgroups in the Cremona group is based on the following result \cite[Lemma 6]{di-perf}.
\begin{lem}
	Let $G\subset\CRk$ be a finite subgroup. Then there exists a $\kk$-rational smooth projective surface $X$, an injective homomorphism \[\iota: G\rightarrow \Aut_{\kk}(X)\] and a birational $G$-equivariant $\kk$-map $\psi: X\dasharrow\Pk$, such that
	\[
	G=\psi\circ\iota (G)\circ{\psi}^{-1}
	\]
\end{lem} 
Of course, the $G$-surface $(X,G,\iota)$ can be replaced by a minimal $\kk$-rational $G$-surface, so there is a natural bijection between the conjugacy classes of finite subgroups $G\subset\CRk$ and birational isomorphism classes of minimal smooth $\kk$-rational $G$-surfaces $(X,G)$. The following result is of crucial importance. Its proof can be found in \cite[Theorem 1G]{isk-2}, \cite[Theorem 5]{di-perf}.
\begin{thm}\label{classification}
Let $(X,G)$ be a minimal geometrically rational $G$-surface over a perfect field $\kk$. Then one of the following two cases occurs:
\begin{description}
\item[C] $X$ admits a conic bundle structure with ${\Pic(X)^G\cong\mathbb{Z}^2}$;
\item[D] $X$ is a del Pezzo surface with ${\Pic(X)^G\cong\mathbb{Z}}$. 
\end{description}
\end{thm}

We will also need an important criterion of $\kk$-rationality, which is due to V. A. Iskovskikh. For more details we refer the reader to \cite[\S 4]{isk-1}.
\begin{thm}\label{IskovskikhCrit}
A minimal geometrically rational surface $X$ over a perfect field $\kk$ is $\kk$-rational if and only if the following two conditions are satisfied:
\begin{description}
\item[(i)] $X(\kk)\ne\varnothing$;
\item[(ii)] $d=K_X^2\geq 5$.
\end{description}
\end{thm}

From now on we set $\kk=\RR$. Denote by $Q_{r,s}$ the smooth quadric hypersurface 
\[
\{[x_1:\ldots :x_{r+s}]: x_1^2+\ldots +x_r^2-x_{r+1}^2-\ldots -x_{r+s}^2=0\}\subset\PP_\RR^{r+s-1}.
\] 
The description of minimal geometrically rational real surfaces with real point is essentially due to A. Comessatti \cite{comm}. Modern proofs can be found in \cite{manin}, \cite{isk-2}, \cite{pol}, \cite{kol}.

\begin{thm}\label{rminimal}
Let $X$  be a minimal geometrically rational real surface with $X(\RR)\ne\varnothing$. Then one and exactly one of the following cases occurs:
\begin{enumerate}
\item $X$ is $\RR$-rational: it is isomorphic to $\PR$, to the quadric $Q_{3,1}$ or to a real Hirzebruch surface $\FF_n$, $n\ne 1$;
\item $X$ is a del Pezzo surface of degree 1 or 2 with $\rho(X)=1$;
\item $X$ admits a minimal conic bundle structure $\pi: X\to \mathbb{P}^1$ with even number of singular fibers $c\geq 4$ and $\rho(X)=2$.
\end{enumerate}
\end{thm}

\begin{rem}
	Here is a simple but important observation. The condition $X(\RR)\ne\varnothing$ implies that $\Pic(\XC)^\Gamma=\Pic(X)$, where $\Gamma=\Gal(\CC/\RR)$ is the Galois group \cite[I, 4.5]{Silhol}. In particular, $\rk\Pic(\XC)^{\Gamma\times G}=\rk\Pic(X)^G$.
\end{rem}

\subsection*{Summary} To sum up, let $G$ be a finite subgroup of odd order in $\CRR$. Then we may assume that $G$ acts on a $\mathbb{R}$-rational surface $X$ making $X$ a $G$-minimal surface (in fact, in Sections \ref{sec:bigdp} and \ref{sec: smalldp} we will need only that $X(\RR)\ne\varnothing$, except the cases $K_X^2=3$ and $K_X^2=2$). 

\subsection{A bit of group theory}

The following facts are well-known. We include some proofs for completeness and the reader’s convenience.

\begin{lem}\label{liftinglemma}
	Let $G$ be a finite group of odd order. Then every faithful projective representation $\theta: G\to\PGL_n(\RR)$, $n\geq 2$, can be lifted to a faithful representation $\widetilde{\theta}: \widetilde{G}\cong G\to\SL_n(\RR)$. 
\end{lem}
\begin{proof}
Since $\PGL_n(\RR)\cong\SL_n(\RR)$ for odd $n$, we will assume that $n$ is even. Consider the following commutative diagram:
	\[\xymatrix@C+1pc{
		& 1\ar[d] & 1\ar[d] & 1\ar[d] &\\
		1\ar[r] & \{\pm 1\}\ar[d]\ar[r] & \RR^*\ar[d]\ar[r] & \RR_{+}\ar[d]\ar[r] & 1\\
		1\ar[r]
		& \SL_n(\RR)\ar[d]^{\gamma}\ar[r]
		& \GL_n(\RR)\ar[d]\ar[r]^{\det}
		& \RR^*\ar[d]\ar[r]
		& 1\\
		1\ar[r]
		& \PSL_n(\RR)\ar[d]\ar[r]^{\alpha}
		& \PGL_n(\RR)\ar[r]\ar[d]
		& \{\pm 1\}\ar[r]\ar[d]
		& 1\\
		& 1 & 1 & 1 &
	}\]
Since $G\subset\PGL_n(\RR)$ is of odd order, we have $G\cong\alpha^{-1}(G)\subset\PSL_n(\RR)$. The group $\gamma^{-1}\circ\alpha^{-1}(G)$ is twice larger than $G$; however, we can find an
isomorphic lift of $G$ as a Hall subgroup therein.
\end{proof}

We use Lemma \ref{liftinglemma} to describe all subgroups of odd order in $\PGL_k(\RR)$ for $k=2,3,4$.
\begin{prop}\label{PGLclassification}
	Let $G$ be a finite group of odd order $n$.
	\begin{enumerate}
		\item If $G\subset\PGL_2(\RR)$ then $G$ is a cyclic group generated by a single matrix
		\[R_2(2\pi/n)=
		\begingroup
		\renewcommand*{\arraystretch}{1.2}
		\begin{pmatrix}
		\cos\frac{2\pi}{n} & \sin\frac{2\pi}{n} \\
		-\sin\frac{2\pi}{n} & \cos\frac{2\pi}{n}
		\end{pmatrix}
		\endgroup
		\]
		\item If $G\subset\PGL_3(\RR)$ then $G$ is a cyclic group generated by a single matrix
		\[
		R_3(2\pi/n)=
		\begingroup
		\renewcommand*{\arraystretch}{1.2}
		\begin{pmatrix}
		1 & 0 & 0 \\ 
		0 & \cos\frac{2\pi}{n} & \sin\frac{2\pi}{n} \\
		0 & -\sin\frac{2\pi}{n} & \cos\frac{2\pi}{n}
		\end{pmatrix}
		\endgroup 
		\]	
		\item If $G\subset\PGL_4(\RR)$, then $G$ can be written as a direct product of at most two cyclic groups of odd orders.
		
	\end{enumerate}
\end{prop}
\begin{proof}
	Applying Lemma \ref{liftinglemma}, we may assume that $G\subset\GL_k(\RR)$, $k=2,3,4$, in the corresponding cases above. Moreover, we may assume that $G\subset\SO_k(\RR)$, $k=2,3,4$, since every real representation of a finite group is equivalent to an orthogonal one and $G$ is of odd order. Recall that any finite subgroup of $\SO_2(\RR)$ is cyclic, while any finite subgroup of $\SO_3(\RR)$ is either cyclic, or dihedral, or one of the symmetry groups of Platonic solids $\mathfrak{A}_4$, $\mathfrak{S}_4$ or $\mathfrak{A}_5$. Now (1) is obvious and to conclude with (2) it remains to notice that the cyclic group of order $n$ acts as rotations in a plane, fixing the axis perpendicular to that plane.
	
	In order to prove (3), we use a well-known fact that $\SO_4(\RR)$ is a double cover of $\SO_3(\RR)\times\SO_3(\RR)$ \cite[Chapter 3, \S 3D]{Hat}. Hence, $G\subset\SO_3(\RR)\times\SO_3(\RR)$. Let $\pi_i$ be the projection on the $i$th component. Then $G\subseteq\pi_1(G)\times\pi_2(G)$, where $\pi_i(G)\subset\SO_3(\RR)$ are cyclic groups of odd order. Thus $G$ itself can be written as a direct product of at most two cyclic groups.
\end{proof}

\section{The conic bundle case}\label{sec:cb}

In this section we prove Theorem \ref{conicbundle}. We first recall what is an elementary transformation of a Hirzebruch surface. 

An {\it elementary transformation} of a comlex Hirzebruch surface $\FF_n$ is the following birational transformation. Let $\upsilon: Y\to \FF_n$ be the blow-up of a point $p$ on a fiber $F$, $\widetilde{F}$ is a strict transform of $F$, $\widetilde{C}_n$ is a strict transform of the $(-n)$-section $C_n\subset\FF_n$ and $E$ is the exceptional divisor. We have $(\widetilde{F})^2=(\upsilon^*F-E)^2=F^2-1=-1$. Then there is a morphism $\psi: Y\to Z$ blowing down $\widetilde{F}$ (over $\CC$). If $p\notin C_n$, then $\widetilde{C}_n^2=C_n^2=-n$ and $\widetilde{C}_n$ intersects $\widetilde{F}$ transversely in exactly one point. Thus $\psi(\widetilde{C}_n)^2=-n+1$ and $Z\cong\FF_{n-1}$. If $p\in C_n$, then $\widetilde{C}_n^2=C_n^2-1=-n-1$, $\widetilde{C}_n\cap\widetilde{F}=\varnothing$, so $\psi(\widetilde{C}_n)^2=-n-1$ and $Z\cong\FF_{n+1}$. The following commutative diagram illustrates these birational transformations:
\[\xymatrix@C+1pc{
		& Y\ar[dl]_{\upsilon}\ar[dr]^{\psi} &\\
		\FF_n\ar[d]\ar@{-->}[rr] && {Z=\FF_{n+1}\ \text{or}\ \FF_{n-1}}\ar[d]\\
		\PP^1\ar@{=}[rr] && \PP^1
	}\]
Note that over $\RR$ we can blow up either a real point or two imaginary conjugate points. For example, the blow-up of two conjugate imaginary points $p,\ \overline{p}\notin C_n\subset\FF_n$ with $n>0$ followed by the contraction of the strict transform of the fibres passing through $p,\ \overline{p}$, gives a birational map $\FF_n\dasharrow\FF_{n-2}$. An analogous procedure for a real point $q\in\FF_n(\RR)$ gives a birational map $\FF_n\dasharrow\FF_{n-1}$.
\begin{rem}
In the language of Sarkisov program these elementary transformations are both Sarkisov links of type II between two Mori fibrations. For more details on factorization of birational maps see \cite{isk-1}, \cite{corti} and \cite{pol} (for the case of real rational surfaces). 
\end{rem}
Now we are ready to prove Theorem \ref{conicbundle}.
\begin{proof}[Proof of Theorem \ref{conicbundle}]
	
Let $X$ be a surface of type {\bf (C)} (see Theorem \ref{classification}). Since $X$ is assumed to be $\mathbb{R}$-rational, we have $X(\mathbb{R})\ne\varnothing$. Thus $C(\mathbb{R})\ne\varnothing$ and $C\cong\PRone$.

We may assume that $X$ is relatively minimal. Indeed, suppose that there is an exceptional curve $E$ whose irreducible components are contained in singular fibers of $\pi$. We have the following two cases: {\bf (a)} $E$ is a real irreducible component of some singular fiber $E+E'$; {\bf (b)}  $E=F+\sigma(F)$, $F\cap\sigma(F)=\varnothing$, where $F+N_1$ and $\sigma(F)+N_2$ are two different singular fibers. Note that $G$-minimality of $X$ implies that there exists $g\in G$ such that $g(E)=E'$ in the case {\bf (a)}, and $g(F)=N_1$ or $g(\sigma(F))=N_1$ in the case {\bf (b)}. In both cases $g$ has an even order, so we obtain a contradiction (cf. \cite[Lemma 5.6]{di}). Indeed, in the case {\bf (a)} we have $g(E')=E$ (as $g$ respects the intersection product), $g^2(E')=E'$ and $\ord g=2n+1$ implies $E=g^{2n+1}(E)=g^{2n}(E')=E'$, a contradiction. Now consider the case when $g(\sigma(F))=N_1$ (if $g(F)=N_1$, we argue as above). Then $g(N_2)=F$, $\sigma(N_2)=N_1$, so $g(N_1)=g(\sigma(N_2))=\sigma(g(N_2))=\sigma(F)$. Thus $g^2(N_1)=g(\sigma(F))=N_1$ and $\ord g=2n+1$ implies $\sigma(F)=g^{2n+1}(\sigma(F))=g^{2n}(N_1)=N_1$, a contradiction.

Therefore $\rho(X)=2$ and $G$ acts trivially on $\Pic(X)$. If $X$ is not minimal, then there is a birational morphism $X\rightarrow X'$, where $X'$ is a del Pezzo surface \cite[Theorem 4]{isk-2}. Since $G$ acts trivially on $\Pic(X)$, this morphism is $G$-equivariant and the assertion follows.

Now let $X$ be a minimal surface. Theorem \ref{rminimal} shows that $X\cong\FF_n$, $n\ne 1$. Denote by $G'$ the image of $G$ in ${\rm Aut}(C)\cong{\rm PGL}_2(\RR)$. Since $G$ is of odd order, $G'$ has to be a cyclic group by Proposition \ref{PGLclassification}. Suppose that $n>0$. We have only two possibilities.

{\bf 1.} $G'\ne\{\rm id\}$. Then we have two $G'$-fixed points $p_1$, $p_2=\sigma(p_1)\in C_\CC\cong\mathbb{P}_\CC^1$, corresponding to $G$-invariant fibres $F_1$ and $F_2=\sigma(F_1)$. Making $G$-equivariant elementary transformations centered at two complex conjugate $G$-fixed points $q_i\in F_i$ not lying on the exceptional $(-n)$-section, we obtain a surface $X'\cong\mathbb{F}_{n-2}$. Proceeding in this way, we come either to $\mathbb{F}_1$, being not $G$-minimal, or $\mathbb{F}_0\cong \PRone\times\PRone$. 

{\bf 2.} $G'=\{\rm id\}$. Then $G$ acts by automorphisms of fibres, which are $\PRone$, so it has two complex conjugate fixed points on each fiber (recall that the order of $G$ is odd). One of these points lies on the $(-n)$-section $C_n$, while the other lies on some $n$-section. However, the $(-n)$-section is clearly $\Gamma$-invariant and intersects each fiber in exactly one point. So, this case does not occur (as complex conjugate points cannot lie both on the same fiber and the $(-n)$-section).
\newline
\newline
Therefore, we may assume that $n=0$, i.e. $X\cong\PRone\times\PRone$. Now we are going to study automorphisms of $\PRone\times\PRone$.
\begin{prop}\label{ConicCyclicGroups}
Assume that $X\cong \PRone\times\PRone$ and $\Pic(X)^G\cong\ZZ^2$. Then \[G\subseteq\big\langle R_2(2\pi/l)\big\rangle\times \big\langle R_2(2\pi/m) \big\rangle\] for some $l,m\in\NN$ (see Proposition \ref{PGLclassification} for the notation).
\end{prop}
\begin{proof}
Recall that
\[
\Aut(\PRone\times\PRone)=({\rm PGL}_2(\RR)\times {\rm PGL}_2(\RR))\rtimes\ZZ/2\ZZ.
\]
As the order of $G$ is odd, $G\subset\PGL_2(\RR)\times\PGL_2(\RR)$. Let
\[
\pi_i: {\rm PGL}_2(\RR)\times {\rm PGL}_2(\RR)\rightarrow{\rm PGL}_2(\RR),\ \ i=1,\ 2,
\]
be the projection on the $i$th component. Then $G\subseteq\pi_1(G)\times\pi_2(G)$ and the assertion follows from Proposition \ref{PGLclassification}.
\end{proof}
\begin{sled}
If the conditions of Proposition \ref{ConicCyclicGroups} are satisifed, then $G$ is a product of at most two cyclic groups.
\end{sled}

Theorem \ref{conicbundle} now is proved.

\end{proof}

\section{del Pezzo surfaces with $K_X^2\geq 5$}\label{sec:bigdp}

Throughout the next two sections $X$ will denote a real del Pezzo surface with $X(\RR)\ne\varnothing$. We shall additionally assume that $X$ is $\RR$-rational in Proposition \ref{prop: dP3 Z3-minim} (see Remark \ref{rem: counterexample}). Note that this automatically holds if $K_X^2\geq 5$ by Iskovskikh's Theorem \ref{IskovskikhCrit}.
\newline
\newline
As we already mentioned in Section \ref{sec:prelim}, if $X$ is a del Pezzo surface, then $\XC$ is isomorphic to one of the following surfaces: $\PC$, $\mathbb{P}^1_{\mathbb{C}} \times \mathbb{P}^1_{\mathbb{C}}$, or $\PC$ blown up in $r\leq 8$ points in general position. We have $d=K_X^2=9-r.$

The following simple lemma will be useful for us.
\begin{lem}\label{mainlemma}
Let $X$ be a real del Pezzo surface containing a $G$-invariant exceptional curve. Then $\rk\Pic(X)^G>1$.
\end{lem}
\begin{proof}
Suppose that $\rk\Pic(X)^G=1$. Denote by $E$ a $G$-invariant exceptional curve. By definition, either $E=L$ is a real $(-1)$-curve, or $E=L+\sigma(L)$, where $L$ is exceptional on $\XC$ and $L\cap\sigma(L)=\varnothing$. Then either $L\sim-aK_X$, or $L+\sigma(L)\sim-aK_X$. Clearly, none of the cases occurs, as $E^2=-1$ in the first case and $E^2=-2$ in the second one.
\end{proof}
\begin{rem}
	A del Pezzo	surface of degree 7 is never $G$-minimal. Indeed, there are three exceptional divisors on this surface forming a chain. The middle one is always defined over $\RR$ and $G$-invariant, so we can contract it.
\end{rem}

The number of $(-1)$-curves on del Pezzo surfaces is classically known \cite[Chapter IV, Theorem 26.2]{cubicforms}. Since this information will be used throughout the paper, we provide it in Table ~\ref{table: delPezzoNumberOfLines}.
\setlength{\extrarowheight}{3pt}
\begin{table}[h!]
\caption{$(-1)$-curves on del Pezzo surfaces}
\label{table: delPezzoNumberOfLines}
\begin{tabular}{|c|c|c|c|c|c|c|llllll}
\cline{1-7}
 $d$ & 1 & 2 & 3 & 4 & 5 & 6 \\ \cline{1-7}
 $\#\ \text{of $(-1)$-curves}$ & 240 & 56 & 27 & 16 & 10 & 6  \\ \cline{1-7}
 \end{tabular}
\end{table}

\subsection{Birational maps between del Pezzo surfaces} We will need the following result about birational maps between two del Pezzo surfaces. It is probably well-known to experts, but we provide the proof for the sake of completeness.

\begin{lem}\label{lem:blow-up-down-dP}
	Let $X$ be a del Pezzo surface.
	\begin{enumerate}
		\item Let $\upsilon: X\to Z$ be a birational morphism. Then $Z$ is a del Pezzo surface.	
		\item Let $\pi: Y\to X$ be the blow-up of any point $p\in X$. Assume that the following conditions are satisfied: (i) $K_X^2>1$; (ii) $p$ does not lie on any $(-1)$-curve; (iii) additionally, $p$ does not lie on the ramification divisor of the double cover $\varphi_{|-K_X|}: X\to\PC$ when $K_X^2=2$. Then $Y$ is a del Pezzo surface with $K_Y^2=K_X^2-1$.
	\end{enumerate}
\end{lem}
\begin{proof}

For (1) see \cite[Chapter IV, Corollary 24.5.2]{cubicforms}. Let us prove (2). We have $K_Y=\pi^*K_X+E$, where $E$ is the exceptional divisor. Obviously, $K_Y^2=K_X^2-1>0$. By the Riemann–Roch theorem, $\dim|-K_Y|\geq K_Y^2>0$. Suppose that there is an irreducible curve $C\subset Y$ with $K_Y\cdot C\geq 0$ and put $R=\pi(C)$. If $R$ is nonsingular at $p$, then $C=\pi^*R-E$, so 
\[
K_Y\cdot C=(\pi^*K_X+E)(\pi^* R-E)=\pi^*K_X\cdot \pi^*R-E^2=K_X\cdot R+1\leq 0,
\]
where the last inequality is caused by ampleness of $-K_X$. We see that $K_Y\cdot C=0$, $K_X\cdot R=-1$. By the Hodge index theorem, $C^2<0$, so by the adjunction formula $C^2=-2$ and $C\cong\PCone$. This means that $R$ is a $(-1)$-curve, a contradiction.
	
Now let $p$ be a singular point of $R$. Note that $R$ must be a component of some divisor $\pi_*(R')$ where $R'\in|-K_Y|$. It is easy to see that $R=\pi_*(R')$ and $p_a(R)=1$, so $p$ is either an ordinary double point or a cusp. Therefore,
\[
C=\pi^*R-2E\sim \pi^*(-K_X)-2E=-K_Y-E.
\]
Thus $-K_Y\cdot C=K_Y^2-1\geq 0$. It follows that $K_Y\cdot C=0$, $K_Y^2=1$ and $K_X^2=2$. We see that $\varphi_{|-K_X|}(R)$ touches the branch curve at $\varphi_{|-K_X|}(p)$, a contradiction.
\end{proof}

\subsection{The Weyl groups}\label{sec:weyl}

There is a powerful tool for studying the geometry of del Pezzo surfaces, namely the {\it Weyl groups}. For convenience of the reader we recall definitions and basic facts (see \cite{cubicforms}, \cite{cag}).

Let $\XC$ be a complex del Pezzo surface of degree $d\leq 6$, obtained by blowing up $\PC$ in $r=9-d$ points. The group $\Pic{\XC}\cong\ZZ^{r+1}$ has a basis $e_0,\ e_1,\ldots,e_r$, where $e_0$ is the pull-back of the class of a line on $\PC$, and $e_i$ are the classes of exceptional curves. Put
\[
\Delta_r=\{s\in\Pic(\XC):\ s^2=-2,\ s\cdot K_{\XC}=0 \}.
\]
Then $\Delta_r$ is a root system in the orthogonal complement to $K_{\XC}^{\bot}\subset\Pic(\XC)\otimes\RR$. As usual, one can associate with $\Delta_r$ the Weyl group $\WW(\Delta_r)$. By definition, $\WW(\Delta_r)$ is the subgroup of $\Ort(K_{\XC}^\bot)$ generated by reflections through the hyperplanes orthogonal to the roots $s\in\Delta_r$. Depending on degree $d$, the type of $\Delta_r$ and the size of $\WW(\Delta_r)$ are the following:

\setlength{\extrarowheight}{3pt}
\begin{table}[h!]
\caption{The Weyl groups}
\label{table:weyltable}
\begin{tabular}{|c|c|c|c|c|c|c|lllllll}
\cline{1-7}
 $d$ & 1 & 2 & 3 & 4 & 5 & 6 \\ \cline{1-7}
 $\Delta_r$ & $\E_8$ & $\E_7$ & $\E_6$ & $\D_5$ & $\A_4$ & $\A_1\times \A_2$  \\ \cline{1-7}
 $|\WW(\Delta_r)|$ & $2^{14}\cdot3^5\cdot5^2\cdot7$ & $2^{10}\cdot3^4\cdot5\cdot 7$ & $2^7\cdot3^4\cdot5$ & $2^7\cdot3\cdot 5$ & $2^3\cdot3\cdot 5$ & 12 \\ \cline{1-7} 
 \end{tabular}
\end{table}
Moreover, there are natural homomorphisms 
\[
\rho: \Aut(\XC) \rightarrow \WW({\Delta_r}),\ \ \eta: \Gamma={\rm Gal}(\CC/\RR)\rightarrow \WW({\Delta_r}),
\]
where $\rho$ is an injection for $d\leq 5$. We denote by $g^*$ the image of $g\in\Gamma\times G$ in the corresponding Weyl group.

Denote by $\mathbb{E}_r$ the sublattice of $\Pic(\XC)$ generated by the root system $\Delta_r$. For an element $g^*\in\WW(\Delta_r)$ denote by $\tr(g^*)$ its trace on $\mathbb{E}_r$. To determine whether a finite group $\Gamma\times G$ acts minimally on $\XC$, we use the well-known formula from the character theory of finite groups
\begin{equation}\label{characterformula}
{\rm rk}\Pic(\XC)^{\Gamma\times G}=1+\frac{1}{|\Gamma\times G|}\sum_{g\in \Gamma\times G}\tr(g^*).
\end{equation}
Thus the group $\Gamma\times G$ acts minimally on $\XC$ if and only if $\sum_{g\in \Gamma\times G}\tr(g^*)=0$. On the other hand, by the Lefschetz fixed point formula (see \cite[Chapter 2, \S 2C]{Hat}) for any $h\in G$ we have,
\begin{equation}\label{LefschetzFixedPointFormula}
\Eu(\XC^h)=\tr(h^*)+3,
\end{equation}
where here and later we denote by $\Eu(\cdot)$ the topological Euler characteristic.
\begin{rem}\label{rem: holom Lef}
	Note that a cyclic group always has a fixed point on a complex rational variety. This follows from the holomorphic Lefschetz fixed-point formula.
\end{rem}
Denote by $\Sp(g^*)$ the set of eigenvalues of $g^*$. For a cyclic group $\Gamma\times G\cong\langle g\rangle_n$ of order $n$ it is very easy to determine whether this group acts minimally on $\XC$.
\begin{lem}\label{cyclicminimality}
A del Pezzo surface $X$ is $\langle g\rangle_n$-minimal if and only if $1\notin\Sp(g^*)$.
\end{lem}
\begin{proof}
According to the formula (\ref{characterformula}), we have to show that the sum of the traces $\tr(g^{*k})$ adds up to 0 if and only if $1\notin\Sp(g^*)$. Let $\lambda_1,\ldots, \lambda_r$ be the eigenvalues of $g^*$. We have
\[
\sum_{k=0}^{n-1}\tr(g^{*k})=\sum_{k=0}^{n-1}\sum_{i=1}^{r}\lambda_i^k=\sum_{i=1}^{r}\sum_{k=0}^{n-1}\lambda_i^k.
\]
It remains to notice that $\sum_{k=0}^{n-1}\lambda_i^k$ equals $n$ for $\lambda_i=1$ and 0 otherwise. 
\end{proof}

\subsection{del Pezzo surfaces of degree $9$}\label{sec:dp9}
Let $X$ be a real del Pezzo surface of degree 9. Then $X$ is a Severi-Brauer variety of dimension 2. As $X(\mathbb{R})\ne\varnothing$, we have $X\cong\PR$ and $G\subset {\rm PGL}_3(\mathbb{R})$. Applying Proposition \ref{PGLclassification}, we obtain the following
\begin{prop}\label{prop: dp9 main prop}
Let $X$ be a real del Pezzo surface of degree 9 and $G$ be a subgroup of odd order $n$ in the automorphism group of $X$. Then $G\subset\PGL_3(\RR)$ and $G$ is isomorphic to a cyclic group of order $n$, generated by $R_3(2\pi/n)$ (see Proposition \ref{PGLclassification} (2)).
\end{prop}

\subsection{del Pezzo surfaces of degree $8$}\label{sec:dp8}
In this subsection $X$ denotes a real del Pezzo surface of degree 8. We shall assume that $\XC\cong\PCone\times\PCone$ (the other surface of degree 8, the blow up of $\PP_\RR^2$ at one point, is never $G$-minimal), so either $X\cong Q_{3,1}$ or $X\cong Q_{2,2}$ \cite[Lemma 1.16]{kol}.
\begin{prop}\label{prop: dP8_main_prop}
Let $X$ be a real del Pezzo surface of degree 8 with ${\rm Pic}(X)^G\cong\mathbb{Z}$, where $G$ is a group of odd order. Then $G$ is linearizable (and hence is cyclic).
\end{prop}
\begin{proof}
Since $G$ is of odd order, the two components of $\XC\cong\mathbb{P}^1_{\mathbb{C}}\times\mathbb{P}^1_{\mathbb{C}}$ are exchanged by the Galois group only. Thus $\Pic(X)^G=\Pic(X)\cong\mathbb{Z}$ and $X$ is $\mathbb{R}$-minimal. Theorem \ref{rminimal} shows, that $X\cong Q_{3,1}$, so $X(\mathbb{R})$ is homeomorphic to a sphere $\SSS^2$. Suppose that $G$ has a real fixed point $p$. Blowing it up and contracting the strict transforms of the lines passing through $p$, we see that our group $G$ is conjugate to a subgroup of $\PGL_3(\RR)$. Thus, $G$ must be cyclic by Proposition \ref{PGLclassification}.

It remains to explain why $G$ always has a real fixed point. First, let us notice that $G$ is a direct product of at most two cyclic groups.
Indeed, any automorphism of $X$ is a restriction of a projective automorphism of $\PP_\RR^3$, so we can identify automorphisms of $X$ with elements of $\PGL_4(\RR)$. By Proposition \ref{PGLclassification}, our group $G$ is a direct product of two cyclic groups, say $G_1\cong\langle g_1\rangle$ and $G_2\cong\langle g_2\rangle$.

Applying the topological Lefschetz fixed point formula, we see that
\[
\Eu((\SSS^2)^{g_1})=\sum_{k\geq 0}\tr_{H_k(\SSS^2,\RR)}g_{1*}=\tr_{H_0(\SSS^2,\RR)}g_{1*}+\tr_{H_2(\SSS^2,\RR)}g_{1*}=2
\]
Here we denote by $g_{1*}$ the induced action on homology. Note that $\tr_{H_0(\SSS^2,\RR)}g_{1*}=1$ as $\SSS^2$ is path-connected, and $\tr_{H_2(\SSS^2,\RR)}g_{1*}=1$ since $g_1$ has an odd order, hence preserves orientation of $\SSS^2$.  Thus, $(\SSS^2)^{g_1}$ consists of two points, say $p$ and $p'$. Then $G_2$ acts on the set $\{p,p'\}$ and the action is trivial, because the order of $G_2$ is odd.  
\end{proof}
\begin{ex}
	One can explicitly write the action of a cyclic group $G$ on the quadric $X\cong Q_{3,1}=\{[x:y:z:w]: x^2+y^2+z^2=w^2\}$ as
	\[
	[x:y:z:w]\mapsto[x\cos\theta+y\sin\theta: -x\sin\theta+y\cos\theta:z:w].
	\]
	This is obviously a rotation around $z$-axis that fixes two points (the North and the South poles) on the sphere.
\end{ex}

\subsection{del Pezzo surfaces of degree $6$}\label{sec:dp6}
In this subsection $X$ denotes a real del Pezzo surface of degree 6. Recall that $\XC$ is isomorphic to the surface obtained by blowing up $\PC$ in three noncollinear points $p_1,p_2,p_3$. The set of $(-1)$-curves on $\XC$ consists of six curves: the exceptional divisors of blow-up $e_i={\pi}^{-1}(p_i)$ and the strict transforms of the lines $d_{12}=\overline{p_1,p_2}$, $d_{13}=\overline{p_1,p_3}$, $d_{23}=\overline{p_2,p_3}$. In the anticanonical embedding $\XC\hookrightarrow\PP_\CC^6$ these exceptional curves form a hexagon of lines. We denote this hexagon by $\Sigma$. Obviously, $\Aut(\XC)$ preserves $\Sigma$, so there is a homomorphism
\[
\rho: \Aut(\XC)\to\Aut(\Sigma)\cong\mathcal{D}_6\cong\mathfrak{S}_3\times\ZZ/2\ZZ,
\]
where $\mathcal{D}_{6}\cong\WW(\A_1\times \A_2)$ is a dihedral group of order $12$ and $\mathfrak{S}_3$ is a symmetric group on 3 letters. The kernel ${\rm Ker}(\rho)$ is isomorphic to the torus $T\cong(\CC^*)^2$ (it comes from an automorphism of $\PC$, that fixes all the points $p_i$). In fact, one can show that $\Aut(\XC)\cong T\rtimes\mathcal{D}_6$. Put $G_T=G\cap T$, $\widehat{G}=\rho(G)$. Then we get a short exact sequence
\begin{equation}\label{eq: dP6_exact_seq}
1\longrightarrow G_T\longrightarrow G\overset{\rho}{\longrightarrow} \widehat{G}\longrightarrow 1\tag{$\star$}
\end{equation}
\begin{prop}\label{prop:dP_6_main_prop}
	Let $X$ be $G$-minimal real del Pezzo surface of degree 6, where the order of $G$ is odd. Then $\widehat{G}\cong\ZZ/3\ZZ$ and the exact sequence \eqref{eq: dP6_exact_seq} splits, i.e. $G\cong G_T\rtimes (\ZZ/3\ZZ)$.
\end{prop}
\begin{proof}
	Note that $\widehat{G}\ne {\rm id}$, since $X$ is not $\mathbb{R}$-minimal by Theorem \ref{rminimal}. Thus $\widehat{G}\cong\ZZ/3\ZZ$. Now let us show that the exact sequence \eqref{eq: dP6_exact_seq} splits. To construct a splitting map $\xi: \widehat{G}\to G$, it suffices to find $h\in G$ such that $\rho(h)$ generates $\widehat{G}\cong\ZZ/3\ZZ$ and $h^3={\rm id}$ (then one can put $\xi(\rho(h))=h$). Since $\rho$ is surjective, we can easily find $h\in G$ such that $\rho(h)$ generates $\widehat{G}$. Let us then show that $h^3={\rm id}$ automatically. Pick up any point $q\in X(\CC)$ which is fixed by $h$ (such a point always exists by Remark \ref{rem: holom Lef}) and blow it up. Note that $q\notin\Sigma$, so by Lemma \ref{lem:blow-up-down-dP}, the obtained surface is a del Pezzo surface of degree 5. Moreover, it has 3 disjoint $(-1)$-curves forming one $\langle h\rangle$-orbit. Blowing this orbit down, we get 3 points on the diagonal of $\PCone\times\PCone$ which are fixed by $h^3$. It follows that $h^3={\rm id}$. Proposition \ref{prop:dP_6_main_prop} now is proved.
\end{proof}

From now on, until the end of this section, we assume that $X$ satisfies the conditions of Proposition \ref{prop:dP_6_main_prop}. Clearly, $\eta(\Gamma)=\ZZ/2\ZZ$ (otherwise all $(-1)$-curves are real, while there is a disconnected orbit of the $G$-action on the set of these curves). There are three principally distinct ways of the Galois group $\Gamma$ action on the hexagon (see Fig. \ref{fig:hexagon}).
\begin{figure}[h!]
\centering
	\begin{subfigure}{.35\textwidth}
	  \centering
	  \includegraphics[width=.6\linewidth]{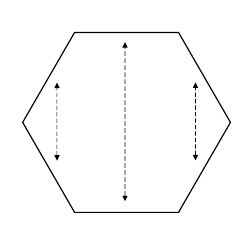}
	  \caption{}
	  \label{fig:a}
	\end{subfigure}%
	\begin{subfigure}{.35\textwidth}
		\centering
		\includegraphics[width=.6\linewidth]{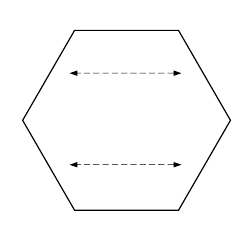}
		\caption{}
		\label{fig:b}
	\end{subfigure}%
	\begin{subfigure}{.35\textwidth}
	  \centering
	  \includegraphics[width=.6\linewidth]{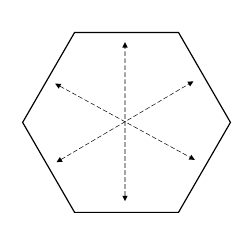}
	  \caption{}
	  \label{fig:c}
	\end{subfigure}
\caption{Galois group acting on the set of exceptional curves}
\label{fig:hexagon}
\end{figure}
Since neither the action of type (A) nor the action of type (B) commutes with ($\ZZ/3\ZZ$)-action, the complex conjugation acts as in Fig. \ref{fig:c}. Then $\sigma^*(e_0)=2e_0-e_1-e_2-e_3=-K_X-e_0$, $\sigma^*(e_0-e_i)=-K_X-e_0-(e_0-e_j-e_k)=e_0-e_i$, so the pencil of conics $|e_0-e_i|$ defines a map $\varphi_i: X\to\PRone$ over $\RR$. The product map $\varphi_1\times\varphi_2\times\varphi_3$ embeds $X$ into $\PRone\times\PRone\times\PRone$ and the image is a divisor of 3-degree (1,1,1). Hence
\begin{equation}\label{eq: trilinearform}
X=\Big\{[x_1:x_2]\times[y_1:y_2]\times[z_1:z_2]\in \PRone\times\PRone\times\PRone: {\rm F}=\sum_{i,j,k=1}^2a_{ijk}x_iy_jz_k=0,\ a_{ijk}\in\RR \Big\}\tag{$\star\star$}.
\end{equation}
According to \cite[Theorem 2]{trilinear}, any binary trilinear form ${\rm F}$ is equivalent over $\RR$ (i.e. there is a nondegenerate change of variables on each factor) to one of the following canonical forms:
\begin{description}
	\item[(a)] $x_1y_1z_1+x_2y_2z_2;$
	\item[(b)] $x_1y_1z_1+x_2y_1z_2+x_2y_2z_1;$
	\item[(c)] $x_1y_1z_1+x_1y_2z_2+x_2y_1z_2-x_2y_2z_1;$
	\item[(d)] $x_1y_1z_1+x_1y_2z_2;$
	\item[(e)] $x_1y_1z_1.$
\end{description}
It is easy to check that forms (b), (d), (e) define singular surfaces, while (a) and (c) are smooth. On the other hand, all $(-1)$-curves on the surface (a) are real, contradicting our observation that $\Gamma$ acts as in Fig. \ref{fig:c}. Thus, we may assume that $X$ is given by the equation (c). 

\begin{rem}\label{rem: topology_of_dP6}
	Let us clarify the topology of $X(\RR)$. A real del Pezzo surface of degree 6 is isomorphic to one of the following surfaces: $\PR$ blown up at $a> 0$ real points and $b\geq 0$ pairs of conjugate points for some $a+2b=3$ (then $X(\RR)\approx\#(a+1)\RR\PP^2$), $Q_{3,1}$ blown up at a pair of conjugate points (so $X(\RR)\approx\SSS^2$), or $Q_{2,2}$ blown up at a pair of conjugate points (then $X(\RR)\approx\TT^2=\SSS^1\times\SSS^1$) \cite[Proposition 5.3]{kol}. As we saw earlier, the complex conjugation acts on the set of $(-1)$-curves as in Fig. \ref{fig:c}. This immediately gives $X(\RR)\approx\TT^2$. Indeed, $X$ does not dominate $\PR$ since there are no real $(-1)$-curves on $X$. On the other hand, $X(\RR)$ cannot be a sphere because otherwise there would be two pairs of conjugate intersecting $(-1)$-curves (as in Fig. \ref{fig:a}).
\end{rem}

\begin{prop}\label{prop: dP6 linearization}
Assume the conditions of Proposition \ref{prop:dP_6_main_prop} are satisfied. Then \[G\cong (\ZZ/n\ZZ\times\ZZ/m\ZZ)\rtimes (\ZZ/3\ZZ)\] for some odd integer numbers $n,m\geq 1$. This group is linearizable if and only if $n=m=1$.
\end{prop}
\begin{proof}
Recall that there is a single isomorphism class of complex del Pezzo surfaces of degree 6, since any three	non-collinear points on $\PC$ are $\PGL_3(\CC)$-equivalent. Thus, we can view $\XC$ as a surface in $\PCone\times\PCone\times\PCone$ defined by the equation $x_1y_1z_1=x_2y_2z_2$. It is a compactification of the standard torus $T=(\CC^*)^2=\{(\lambda_1,\lambda_2,\lambda_3)\in(\CC^*)^3: \lambda_1\lambda_2\lambda_3=1\}$ whose real form\footnote{Recall that if a pair $(X,\sigma)$ of a complex projective variety and an antiholomorphic involution $\sigma$ on $X$ is given, then $Y=X/\langle\sigma\rangle$ is a scheme over $\RR$ such that $Y_\CC\cong X$. This scheme $Y$ is called a real form of $X$.}, which we denote by $T_\RR$, is $\SO_2(\RR)\times\SO_2(\RR)$ (see Remark \ref{rem: topology_of_dP6}; on the tori over reals see \cite[10.1]{Voskr}). In particular, $G_T\cong\ZZ/n\ZZ\times\ZZ/m\ZZ$ for some odd integer numbers $n$, $m$. 
	
Now let us prove the second part of the Proposition. Assume that $G$ is linearizable. Then $G$ is a cyclic group (see Proposition \ref{prop: dp9 main prop}), so $G_T$ must be a cyclic group of order coprime to 3, and the action of $\widehat{G}$ on $G_T$ must be trivial. It is clear from the description above that there are exactly three $\widehat{G}$-fixed points on the torus $T(\CC)$, namely the fixed points of the transformation $\lambda_1\mapsto\lambda_2\mapsto\lambda_3$. Thus, $G_T\cong\ZZ/3\ZZ$. In particular, we see that $G$ is not cyclic, hence not linearizable.

Now assume that $n=m=1$, i.e. $G\cong\ZZ/3\ZZ$ and a generator of $G$ acts on $\Sigma$ ``rotating'' it by $2\pi/3$. Let us denote this generator by $\tau$. We claim that $\tau$ has a {\it real} fixed point. Clearly, a fixed point cannot lie on $\Sigma$, since $\tau$ rotates the hexagon by $2\pi/3$. Besides, $G$ has a zero-dimensional fixed point locus on $\XC$ (otherwise, the curve of fixed points meets $\Sigma$, which is an ample divisor).

Applying the Lefschetz fixed point formula, we obtain
\[
\Eu(\XC^\tau)=\sum_{k=0}^4(-1)^k \tr_{H^k (X,\RR)}(\tau^*)=\tr_{H^0(X,\RR)}(\tau^*)+\tr_{\Pic(\XC)}(\tau^*)+\tr_{H^4(X,\RR)}(\tau^*)=2+\tr_{\Pic(\XC)}(\tau^*).
\]
As $\tau$ acts by $e_0\mapsto e_0$, $e_1\mapsto e_2$, $e_2\mapsto e_3$, $e_3\mapsto e_1$, we have $\tr_{\Pic(\XC)}(\tau^*)=1$ and $\Eu(\XC^\tau)=3$. Since the fixed point locus is zero-dimensional, the number of fixed points equals the Lefschetz number. Finally, at least one of those three fixed points must be real. 

Denote by $Y$ the blow-up of this point. By Lemma \ref{lem:blow-up-down-dP}, $Y$ is a del Pezzo surface of degree 5. Topologically, each blowing up at a real point is connected sum with $\RR\PP^2$, so $Y(\RR)\approx \TT^2\#\RR\PP^2$ by Remark \ref{rem: topology_of_dP6}. Since $Y(\RR)$ is nonorientable and $\Eu(\TT^2\#\RR\PP^2)=\Eu(\TT^2)+\Eu(\RR\PP^2)-2=-1$, we get $Y(\RR)\approx \#3\RR\PP^2$. Note that there are 3 disjoint real $(-1)$-curves after blow-up. Blowing them down, we obtain a del Pezzo surface $Z$ of degree 8 either with $Z(\RR)\approx \SSS^2$ (then $Y$ is isomorphic to $Z\cong Q_{3,1}$ blown up at 3 real points), or $Z(\RR)\cong \TT^2$ (then $Y$ is isomorphic to $Z\cong Q_{2,2}\cong\PRone\times\PRone$ blown up at one real point and a pair of complex conjugate). In the first case, as we saw earlier, $G$ has a real fixed point on $Q_{3,1}$. The second case is just impossible\footnote{Basically, it is already clear from the configuration of lines on $Y$.} since $G$ must fix all 3 points on the diagonal of $Z\cong Q_{2,2}$ ($G$ cannot switch complex conjugate points), hence must be trivial. So, $G$ has a real fixed point on $Z\cong Q_{3,1}$. Blowing it up and contracting the strict transforms of the lines passing through it, we conjugate $G$ to a subgroup of $\PGL_3(\RR)$. 
\end{proof}

\begin{rem}[\bf Two non-conjugate embeddings of $(\ZZ/3\ZZ)^2$ into $\CRR$]
	We have the following two actions of the group $G\cong(\ZZ/3\ZZ)^2$ on the real rational surfaces $X$ and $Y$:
	\begin{description}
		\item[1] $X$ is a conic bundle $\PRone\times\PRone$ with $\rk\Pic(X)^G=2$ (case (2) of Theorem \ref{conicbundle});
		\item[2] $Y$ is a real del Pezzo of degree 6 with $\rk\Pic(Y)^G=1$ (corresponding to the case when $\widehat{G}$ acts trivially on $G_T$).
	\end{description}
	The images of $G$ under these embeddings are not conjugate in $\CRR$. It can be deduced from the classification of links \cite[Theorem 2.6]{isk-1}. Namely, we need a link of type (I) or (III) starting at a point of degree 3, but the cited theorem shows that there are no such links. 
\end{rem}

\begin{ex}
	There is an infinite series of non-linearizable subgroups of odd order in $\CRR$. For example, take $G_T$ to be the group of points of order $n$ in $T_\RR(\RR)\cong\SO_2(\RR)\times\SO_2(\RR)$. Clearly, this group is isomorphic to $(\ZZ/n\ZZ)^2$ and normalized by $\widehat{G}$. The whole group $G\cong (\ZZ/n\ZZ)^2\rtimes(\ZZ/3\ZZ)$ is not linearizable by Proposition \ref{prop: dP6 linearization}.
\end{ex}

\begin{ex}
	Let us give an explicit example of an automorphism $\tau\in\Aut(X)$ such that $\langle\tau\rangle\cong\ZZ/3\ZZ$ acts minimally on the surface $X$ given by the polynomial (c). Namely, consider the map
	\[
	\tau_0\in\Aut(\PRone\times\PRone\times\PRone),\ \ \tau_0: [x_1:x_2]\times[y_1:y_2]\times[z_1:z_2]\mapsto [y_1:y_2]\times[z_1:-z_2]\times[x_1:-x_2]
	\]
	and denote by $\tau$ its restriction to $X$. Let $L_k^{\pm}$ denote the equations of the two (complex conjugate) singular fibres of the conic bundle obtained by projecting to the $k$-th factor in \eqref{eq: trilinearform}. The equations $L_k^{\pm}$ are:
	\begin{itemize}
		\item[] $L_1^\pm: y_1z_1+y_2z_2\pm i(y_1z_2-y_2z_1)=0$; 
		\item[] $L_2^\pm: x_1z_1+x_2z_2\pm i(x_1z_2-x_2z_1)=0$;
		\item[] $L_3^\pm: x_1y_1-x_2y_2\pm i(x_1y_2+x_2y_1)=0$.
	\end{itemize}
	It is immediately checked that $\tau^3={\rm id}$, $\tau(L_1^\pm)=L_2^\pm$, $\tau(L_2^\pm)=L_3^\pm$ and $\tau(L_3^\pm)=L_1^\pm$. The fixed locus consists of three points $[t:1]\times[t:1]\times[-t:1]$, $t\in\{0,\pm\sqrt{3}\}$.
\end{ex}

\subsection{del Pezzo surfaces of degree $5$}\label{sec:dp5}

In this subsection $X$ denotes a real del Pezzo surface of degree 5. Recall that $\XC$ is the blow-up of $\PC$ at four points $p_1,p_2,p_3,p_4$ in general position. Let $e_i$ be the exceptional divisor over the point $p_i$ and $d_{ij}$ be the proper transform of the line passing through the points $p_i$ and $p_j$. It is classically known that $\Aut(\XC)\cong\WW(\A_4)\cong\mathfrak{S}_5$ \cite[8.5.4]{cag}. Thus either $G\cong \ZZ/3\ZZ$, or $G\cong \ZZ/5\ZZ$. 

\begin{prop}
Let $G\cong \ZZ/3\ZZ$. Then $X$ is not $G$-minimal.
\end{prop}

\begin{proof}

Recall that there are exactly ten $(-1)$-curves on $\XC$. We claim that there is exactly one $G$-invariant $(-1)$-curve on $X$ (in particular, this curve is real). Indeed, one can see it on the graph of exceptional curves on $\XC$. The incidence graph of the set of these 10 lines is the famous Petersen
graph (see Fig. \ref{fig:3DPetersen} for its <<3D>> form). Our group $G\cong\ZZ/3\ZZ$ acts on this tetrahedron by simply rotating it. It remains to use Lemma \ref{mainlemma}.

\begin{figure}[h!]
	\centering
	\begin{subfigure}{.49\textwidth}
		\centering
		\includegraphics[width=.64\linewidth]{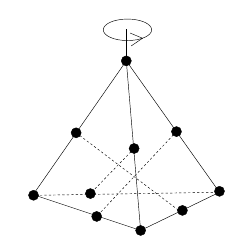}
		\caption{``3D'' form}
		\label{fig:3DPetersen}
	\end{subfigure}
	\begin{subfigure}{.49\textwidth}
		\centering
		\includegraphics[width=.64\linewidth]{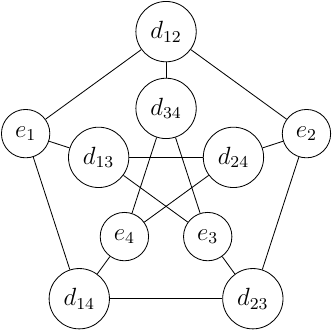}
		\caption{``Classic'' form}
		\label{fig:ClassicPetersen}
	\end{subfigure}
	\caption{Graph of $(-1)$-curves on del Pezzo surface of degree 5}
	\label{pic:hexagon}
\end{figure}

\end{proof}

\begin{lem}\label{lem:DP5FixedPoints}
	Let $X$ be a real del Pezzo surface of degree 5 and $g\in\Aut(X)$ be an automorphism of order 5 acting minimally on $X$. Then $g$ has exactly two fixed points on $\XC$ and these points do not lie on the $(-1)$-curves on $\XC$.
\end{lem}
\begin{proof}
	The $\langle g\rangle_5$-minimality assumption implies that all the $(-1)$-curves on $X$ are real, since the total number of real $(-1)$-curves can be equal to 2, 4 or 10 \cite[Corollary 5.4]{kol}. Now look at the Petersen graph Fig. \ref{fig:ClassicPetersen}. One can check that the five $(-1)$-curves from each $\langle g\rangle_5$-orbit form a pentagon (there are no $g$-invariant $(-1)$-curves). Without loss of generality, we may assume that these orbits are $\{e_1,d_{14},d_{23},e_2,d_{12}\}$ and $\{d_{13},e_4,e_3,d_{24},d_{34}\}$. Obviously, $g$ permutes $(-1)$-curves in the following way: $e_1\mapsto d_{14}= e_0-e_1-e_4$, $e_2\mapsto d_{12}= e_0-e_1-e_2$, $e_3\mapsto d_{24}= e_0-e_2-e_4$, $e_4\mapsto e_3$. In particular, if a fixed point exists, then it cannot lie on any $(-1)$-curve. If $e_0\mapsto w$, then $K_{\XC}= -3e_0+e_1+e_2+e_3+e_4= -3w+(e_0-e_1-e_4)+(e_0-e_1-e_2)+(e_0-e_2-e_4)+e_3$ (since the canonical class is $g$-invariant), so $e_0\mapsto w=2e_0-e_1-e_2-e_4$. Therefore, $\tr_{\Pic(\XC)}(g^*)=0$. As in the previous section, it is easy to see that the fixed point locus is discrete. It remains to apply the Lefschetz fixed point formula:
	\[
	\Eu(\XC^g)=\tr_{H^0(X,\CC)}(g^*)+\tr_{H^4(X,\CC)}(g^*)=2
	\]
\end{proof}

\begin{lem}\label{lem:FromDP5ToDP3}
	Let $X$ be a del Pezzo surface of degree 5 and $\pi: Y\to X$ is the blow-up of two points $q_1,\ q_2\in X$ lying neither on any exceptional curve, nor on any conic (by a conic we mean a conic in the anticanonical embedding of $X$). Then $Y$ is a del Pezzo surface of degree 3.
\end{lem}
\begin{proof}
	To show that $-K_Y$ is ample, we use the Nakai–Moishezon criterion. First, note that $(-K_Y)^2=K_X^2-2=3$. By Riemann-Roch,
	\[
	\dim|-K_Y|\geq\frac{1}{2}((-K_Y)^2-(-K_Y\cdot K_Y))=K_Y^2=3,
	\]
	so $|-K_Y|\ne\varnothing$. Assume that there is an irreducible curve $C\subset Y$ with $-K_Y\cdot C\leq 0$. Clearly, there exists a linear system $\LL\subset |-K_Y|$ of dimension $\geq 2$ such that $C\subseteq F$, where $F$ is the fixed part of $\LL$. Let $\MM=\LL-F$ be the mobile part. Note that $C\nsubseteq\Exc(\pi)$ (since every exceptional curve has positive intersection with $-K_Y$), so $C'=\pi_*C$ is a curve. Put $\LL'=\pi_*\LL$, $F'=\pi_*F$, $\MM'=\pi_*\MM$. Then $\LL'=F'+\MM'\subset |-K_X|$ and $C'\subseteq F'\subset\Bs(\LL')$. Obviously, $p_1,\ p_2\in\Bs(\LL')\backslash C'$. Thus $\Bs(\LL')\subset X\cap\PP^2$ (we identify $X$ with its anticanonical model in $\PP^5$). But the homogeneous ideal of $X$ is generated by five linearly independent quadrics \cite[8.5.2]{cag}, so $p_1,\ p_2$ lie on the curve of degree $\leq 2$, a contradiction.
	
\end{proof}

Now we are ready to prove the main result of this section.

\begin{prop}\label{prop: dp5 order 5}
	Let $X$ be a real del Pezzo surface of degree 5 and $G\subset\Aut(X)$ is a subgroup of order 5 acting minimally on $X$. Then $G$ is linearizable.
\end{prop}
\begin{proof}
	According to Lemma \ref{lem:DP5FixedPoints} $G$ has two fixed points on $\XC$ not lying on the $(-1)$-curves. Denote these points by $q_1$ and $q_2$, and let $Y$ be the blown up surface $\Bl_{q_1,q_2}(X)$. We claim that $Y$ is a del Pezzo surface of degree 3. According to Lemma \ref{lem:FromDP5ToDP3}, we have to show that these points do not lie on any conic.
	
	Suppose that $q_1,\ q_2\in Q$, where $Q$ is a conic. Note that $Q$ is unique. Indeed, assume that $q_1,\ q_2\in Q\cap Q'$, where $Q'$ is another conic. Blowing up $X$ at $q_1$, we get a del Pezzo surface $X'$ of degree 4 with 3 lines forming a triangle (possibly degenerated). On the other hand, it is well-known that there cannot be such triangles on $X'$ (a del Pezzo surface of degree 4 is a complete intersection of two quadrics in $\PP^4$ \cite[Theorem 8.6.2]{cag}) . Finally, $Q$ is obviously $\Gamma\times G$-invariant, so $Q\sim -aK_{\XC}$, $a\in\ZZ$. Multiplying by $-K_{\XC}$, we get $5a=-K_{\XC}\cdot Q=2$, which is impossible.
	
	If $q_1=\sigma(q_2)$, then we have a smooth real cubic surface $Y=\Bl_{q_1,q_2}(X)$ with two skew complex conjugate $G$-invariant lines, say $\ell_1$ and $\ell_2$. It is classically known that such a surface is $G$-birational to a del Pezzo surface $Q_{3,1}$ via birational map
	\begin{equation}\label{eq: bir trivial cubic}
	\varphi: \ell_1\times\ell_2\dashrightarrow Y,\ \ (p_1,p_2)\mapsto\text{third intersection point}\ Y\cap\overline{p_1p_2}.
	\end{equation}
	Finally, the obtained action of $G$ on $Q_{3,1}$ is linearizable by Proposition \ref{prop: dP8_main_prop}. Now assume that $q_1,\ q_2\in X(\RR)$ and let $X\subset\PP(V)$, where $V$ is a 6-dimensional vector space. Then we have a commutative diagram
	\[\xymatrix@C+1pc{
		X\ar[r]^{g}\ar@{-->}[d]_{\pi}  & X\ar@{-->}[d]^{\pi}\\
		\PP_\RR^2\ar[r]^{\widetilde{g}} & \PP_\RR^2  
	}\]
	where $g$ is a generator of $G$, $\pi$ is the projection from $\PP(T_{q_1}X)$ and $\widetilde{g}$ is an automorphism of $\PP_\RR^2\cong\PP(V/T_{q_1})$ induced by $g$. It is not hard to show that $\pi$ is birational, see \cite{bovblanc}. So, $G$ is linearizable.
\end{proof}
\begin{ex}[see \cite{df} or \cite{bovblanc}]
	Let $X$ be a surface obtained from $\PR$ by blowing up four real points $p_1=[1:0:0]$, $p_2=[0:1:0]$, $p_3=[0:0:1]$, $p_4=[1:1:1]$. Consider the transformation $g\in\Aut(X)$ of order 5 defined as the lift to $X$ of the birational map
	\[
	g_0:\PR\dashrightarrow\PR,\ \ g_0: [x:y:z]\mapsto[x(z-y):z(x-y):xz].
	\]
	This map has exactly two real fixed points $[\alpha:1:\alpha^2]$, where $\alpha=(1\pm\sqrt{5})/2$ (which give us two real fixed points on $X$). It can be checked that $g$ is conjugate by some real involution to the linear automorphism of $\PR$ (see \cite{bovblanc} for explicit formulas).
\end{ex}

\section{del Pezzo surfaces with $K_X^2\leq 4$}\label{sec: smalldp}

In the next four sections we use the known classification of conjugacy classes in the Weyl
groups. These classes are indexed by {\it Carter graphs}\footnote{We follow the terminology of \cite{di} and refer the reader to the original paper \cite{weyl} for details. All tables of conjugacy classes in sections \ref{sec:dp3}-\ref{sec:dp1} and Appendix \ref{appendix} are cribbed from \cite{weyl}.}. In particular, the Carter graph determines the characteristic polynomial of an element from a given class and its trace on $K_{\XC}^{\bot}$. 

\subsection{del Pezzo surfaces of degree $4$}\label{sec:dp4}

Again, consider representation in the Weyl group:
\[
\eta\times\rho: \Gamma\times G\rightarrow \WW(\D_5)\cong(\ZZ/2\ZZ)^4\rtimes\mathfrak{S}_5.
\]
\begin{prop}\label{prop: dP4 main prop}
Let $X$ be a real del Pezzo surface of degree 4 and $G\subset\Aut(X)$ be a subgroup of odd order. Then $\rk\Pic(X)^G>1$.
\end{prop}
\begin{proof}
Assume that $\rk\Pic(X)^G=1$. Since the order of $G$ is odd, either $G\cong\ZZ/3\ZZ$ or $G\cong\ZZ/5\ZZ$. It is well known that the number $N$ of real $(-1)$-curves on a real del Pezzo surface of degree 4 can be equal to $0,\ 4,\ 8$ or 16 \cite[Table 2]{Wall}. However, under our assumptions on $X$, we have $N=0$ (otherwise there exists $G$-invariant $(-1)$-curve, contradicting Lemma \ref{mainlemma}). In particular, $\eta(\Gamma)\ne {\rm id}$. On the other hand, $\sigma^*\ne {\rm id}$ implies that $G\ncong\ZZ/5\ZZ$, as there are no elements of order 10 in $\WW(\D_5)$ (see \cite[Section 6.4, Table 3]{di}).  

It remains to consider the case $G=\langle g\rangle_3\cong\ZZ/3\ZZ$. The only conjugacy class of elements of order 3 in $\WW(\D_5)$ is the class of type $A_2$ \cite[Section 6.4, Table 3]{di}, so $\Sp(g^*)=\{1,1,1,\omega_3,\conj{\omega}_3\}$ (here and later $\omega_n$ denotes a primitive $n$-th root of unity). As $g^*$ and $\sigma^*$ commute, they are simultaneously diagonalizable and $\Sp(g\circ\sigma)^*=\{\pm 1,\pm 1,\pm 1,\pm\omega_3,\pm\overline{\omega}_3\}$ (as for the sign combination, note that a priori we require only that the corresponding minimal polynomial has real coefficients). However, there are no involutions in $\WW(\D_5)$ which act as $-{\rm id}$ in $\mathbb{E}_5$. Moreover, since $\XC$ is $\langle g\circ\sigma\rangle$-minimal, $1\notin\Sp(g\circ\sigma)^*$ by Lemma \ref{cyclicminimality}. Thus $\Sp(g\circ\sigma)^*=\{- 1,- 1,- 1,\omega_3,\conj{\omega}_3\}$, and $\tr(g\circ\sigma)^*=-4$. However, Table 3 from the loc. cit. shows that there are no such elements of order 6 in $\WW(\D_5)$.
\end{proof}

\subsection{del Pezzo surfaces of degree $3$}\label{sec:dp3} 

Throughout this subsection $X$ denotes a real del Pezzo surface of degree 3. Recall that $X$ is a cubic surface in $\PP_\RR^3$. Since the linear system $|-K_X|=|\mathcal{O}_X(1)|$ is $G$-invariant, any automorphism of $X$ is a restriction of a projective automorphism, so we may identify automorphisms of $X$ with elements of $\PGL_4(\RR)$. 

For real del Pezzo surfaces of degree 3 one can prove the followning useful lemma (note that, unlike Lemma \ref{mainlemma}, it deals with $(-1)$-curves on a {\it complex} surface):
\begin{lem}\label{lem:dP3minimality}
	Let $X$ be a real del Pezzo surface of degree 3 and suppose that there is a $G$-invariant $(-1)$-curve on $\XC$. Then $X$ is not $G$-minimal.
\end{lem}
\begin{proof}
	Assume that the contrary holds and $L\subset\XC$ is such a curve. By Lemma \ref{mainlemma}, it suffices to consider the case $L\ne\sigma(L)$. Note that $L\cap\sigma(L)\ne\varnothing$ (otherwise we have a $G$-invariant exceptional curve $L+\sigma(L)$ on $X$). Denote by $\Pi$ the $G$-invariant plane in $\PP_\CC^3$ spanned by $L$ and $\sigma(L)$. Then $\Pi\cap\XC=\{L,\sigma(L), M\}$ where $M$ is a real line. Obviously, $M$ must be $G$-invariant which contradicts the $G$-minimality assumption.
\end{proof}

According to Proposition \ref{PGLclassification}, $G$ can be written as a direct product of at most two cyclic groups. On the other hand, there is an injective homomorphism 
\[
\rho: G\rightarrow \WW(\E_6),
\]
hence $|G|=3^k5^l$, $k\leq 4,\ l\leq 1$. If $k=0$, then there exists a $G$-invariant $(-1)$-curve on $\XC$ (as the total number of $(-1)$-curves is 27). Thus $X$ is not $G$-minimal by Lemma \ref{lem:dP3minimality}. Note that there are no elements of order 15 (hence $l=0$), 27 and 81 in $\WW(\E_6)$. We see that $G$ is isomorphic to one of the following groups: 
\[
\ZZ/3\ZZ,\ (\ZZ/3\ZZ)^2,\ \ZZ/9\ZZ,\ (\ZZ/9\ZZ)^2,\ \ZZ/3\ZZ\times\ZZ/9\ZZ.
\]
Denote by $\diag[\alpha:\beta:\gamma:\delta]$ the projective automorphism
\[
[x:y:z:w]\mapsto[\alpha x:\beta y:\gamma z:\delta w],\ \ \ \alpha,\beta,\gamma,\delta\in\CC^*.
\]
Let $g\in{\rm PGL}_4(\RR)$ be an element of order 3. Denote by $\Fix(g,Y)$ the fixed locus of $g$, viewed as an automorphism of $Y$, where $Y$ is $\mathbb{P}_\CC^3$ or $\XC$. Obviously, $\Fix(g,\XC)=\Fix(g,\PP_\CC^3)\cap\XC$. 

\begin{prop}
Let $X$ be a real $G$-minimal del Pezzo surface of degree 3. Then $G$ is not isomorphic to any of the following groups: $(\ZZ/3\ZZ)^2,\ \ZZ/9\ZZ,\ (\ZZ/9\ZZ)^2,\ \ZZ/3\ZZ\times\ZZ/9\ZZ$.
\end{prop}
\begin{proof}
It is well-known that a smooth cubic surface over $\RR$ has $N=3,\ 7,\ 15$ or $27$ real lines (see e.g. \cite[VI, 5.4]{Silhol}, although the result goes back to L. Schl\"{a}fli  and L. Cremona). Clearly, $N\ne 7$ under our assumptions on $X$ (otherwise there would be at least one $G$-invariant $(-1)$-cuve on $X$). Suppose that $G$ is isomorphic to one of the groups listed above. Let us consider the remaining cases for $N$.

{\bf Case $N=3$}. We may assume that there are no $G$-invariant lines on $X$. Thus we have a $G$-orbit consisting of 3 real lines, say $\ell_1,\ \ell_2,\ \ell_3$. Denote by $G_0$ the stabilizer subgroup of $\ell_1$. Obviously, $G_0$ is nontrivial and stabilizes the whole orbit (because $G$ is abelian). Since $X$ is $G$-minimal, the lines $\ell_1,\ \ell_2,\ \ell_3$ cannot be disjoint, so they either determine a triangle, or intersect at a single Eckardt point. Let us show that in both cases $g_0$ must be trivial.

Indeed, in both cases we have a projective automorphism $g_0$ which stabilizes a real line, say $\ell_1$, and fixes a real point $p=\ell_1\cap\ell_2$. Restricting $g_0$ to $\ell_1$, we get an automorphism $h=g_0|_{\ell_1}$ of $\PP^1_\RR$ with a real fixed point. Thus either $h$ has order 2 and we obtain a contradiction (the order of $h$ must divide the order of $g_0$), or $h$ fixes $\ell_1$ pointwise. In the latter case $h$ fixes $\ell_2$ pointwise too ($G$ is abelian and $\ell_1$, $\ell_2$ lie in the same orbit), so $g_0$ fixes pointwise the plane in $\PP_\RR^3$ spanned by $\ell_1$ and $\ell_2$, hence must be a reflection.

Therefore $G_0\cong\{\rm id\}$ and $G\cong\ZZ/3\ZZ$, a contradiction.

{\bf Case $N=15$}. Consider the action of $G$ on the set of real lines on $X$. It is easy to see that there must be a $G$-orbit of cardinality 3 (or a $G$-invariant line). As we saw in the previous case, this is impossible.

{\bf Case $N=27$}. Then the Galois group $\Gamma$ acts trivially on $\Pic(\XC)$ and $\XC$ is a $G$-minimal surface. Take $g\in G$. If the order of $g$ is 9, then $\tr(g^*)=0$ (there is a single conjugacy class in $\WW(\E_6)$, see Table \ref{table:weylE6}). If the order of $g$ is 3, then $\tr(g^*)\geq 0$. In fact, Table \ref{table:weylE6} shows that the only negative value of $\tr(g^*)$ is $-3$, so $\Eu(\Fix(g,\XC))=0$. Clearly, $\Fix(g,\XC)$ is an elliptic curve, and $\Fix(g,\PP_\CC^3)$ is a plane, so $g$ has an eigenvalue of multiplicity 3. Since the eigenvalues of $g$ are not all equal to $\pm 1$, its characteristic polynomial cannot belong to $\RR[t]$, a contradiction.

We see that $\sum_{g\in G}\tr(g^*)\ne 0$ (as $\tr({\rm id}^*)\ne 0$). So, $\XC$ is not $G$-minimal, a contradiction.

\end{proof}

\begin{prop}\label{prop: dP3 Z3-minim}
	Let $X$ be a real $\RR$-rational del Pezzo surface of degree 3, and $G\cong\ZZ/3\ZZ$. Then $X$ is not $G$-minimal.
\end{prop}
\begin{proof}
Let $g$ be a generator of $G$. Table \ref{table:weylE6} shows that $\tr(g^*)\in\{-3,0,3\}$. As we saw above, $\tr(g^*)\ne -3$, as $g$ is defined over $\RR$. In the remaining two cases we see from the same table that $g$ has some eigenvalues equal to 1, so the complex involution $\sigma$ maps nontrivially to $\WW(\E_6)$ by Lemma \ref{cyclicminimality}.

\begin{longtable}{|c|c|c|c|}
	\caption{Elements of order 2, 3, 6 and 9 in $\WW(\E_6)$}
	\label{table:weylE6} \\
	\hline
	{Order} & {Carter graph} & {Characteristic polynomial} & $\tr$ \\ \hline
	
	{2} & {$A_1$} & {$p_1(t-1)^5$} & {4}  \\ \hline
	{2} & {$A_1^2$} & {$p_1^2(t-1)^4$} & {2}   \\ \hline
	{2} & {$A_1^3$} & {$p_1^3(t-1)^3$} & {0}   \\ \hline
	{2} & {$A_1^4$} & {$p_1^4(t-1)^2$} & {$-2$}  \\ \hline
	
	{3} & {$A_2$} & {$(t^2+t+1)(t-1)^4$} & {3}   \\ \hline
	{3} & {$A_2^2$} & {$(t^2+t+1)^2(t-1)^2$} & {0}   \\ \hline
	{3} & {$A_2^3$} & {$(t^2+t+1)^3$} & {$-3$}   \\ \hline
	
	{6} & {$E_6(a_2)$} & {$(t^2+t+1)(t^2-t+1)^2$} & {1}   \\ \hline
	{6} & {$D_4$} & {$(t+1)(t^3+1)(t-1)^2$} & {1}  \\ \hline
	{6} & {$A_1\times A_5$} & {$(t+1)(t^5+t^4+t^3+t^2+t+1)$} & {$-2$}  \\ \hline
	{6} & {$A_1^2\times A_2$} & {$(t+1)^2(t^2+t+1)(t-1)^2$} & {$-1$}  \\ \hline
	{6} & {$A_1\times A_2$} & {$(t+1)(t^2+t+1)(t-1)^3$} & {1}   \\ \hline
	{6} & {$A_1\times A_2^2$} & {$(t+1)(t^2+t+1)^2(t-1)$} & {$-2$}   \\ \hline
	{6} & {$A_5$} & {$(t^5+t^4+t^3+t^2+t+1)(t-1)$} & {0}   \\ \hline
	{9} & {$E_6(a_1)$} & {$t^6+t^3+1$} & {0}  \\ \hline

\end{longtable}

Consider the case $\tr(g^*)=3$ first. We have $\Sp(g^*)=\{1,1,1,1,\omega_3,\overline{\omega}_3\}$, so, as in the previous section, we get $\Sp(g\circ\sigma)^*=\{\pm 1,\pm 1,\pm 1,\pm 1,\pm\omega_3,\pm\overline{\omega}_3\}$. Since $\XC$ is $\langle g\circ\sigma\rangle$-minimal, 
\[\Sp(g\circ\sigma)^*=\{- 1,- 1,- 1,- 1,\pm\omega_3,\pm\overline{\omega}_3\},\]
by Lemma \ref{cyclicminimality}. Thus $\tr(g\circ\sigma)^*\in\{-3,-5\}$. Table \ref{table:weylE6} shows that there are no such elements in $\WW(\E_6)$.

Now let $\tr(g^*)=0$. In this case $\Sp(g^*)=\{1,1,\omega_3,\overline{\omega}_3,\omega_3,\overline{\omega}_3\}$. We have the following possibilities for $\Sp(g\circ\sigma)^*$:
\begin{table}[H]
\begin{tabular}{|c|c|c|ll}
\cline{1-3}
\text{Eigenvalues} & \text{Characteristic polynomial} & $\tr(\tau\circ\sigma)^*$ \\ \cline{1-3}
  $-1,-1,\omega_3,\overline{\omega}_3,\omega_3,\overline{\omega}_3$ & $(t+1)^2(t^2+t+1)^2$ & $-4$  \\ \cline{1-3}
  $-1,-1,-\omega_3,-\overline{\omega}_3,\omega_3,\overline{\omega}_3$ & $(t+1)^2(t^2+t+1)(t^2-t+1)$ & $-2$  \\ \cline{1-3}
  $-1,-1,-\omega_3,-\overline{\omega}_3,-\omega_3,-\overline{\omega}_3$ & $(t+1)^2(t^2-t+1)^2$ & $0$  \\ \cline{1-3}
\end{tabular}
\end{table}
Thus $(g\circ\sigma)^*$ belongs to the class $A_1\times A_5$. Moreover, $\Sp(\sigma^*)=\{-1,-1,-1,-1,1,1\}$, and $\sigma^*$ belongs to the class $A_1^4$. It can be shown that there are exactly 3 real $(-1)$-curves on $X$ in this case, and $X(\RR)\approx\SSS^2\sqcup\RR\PP^2$ \cite[Table 2]{Wall}. In particular, $X$ is not $\RR$-rational, a contradiction.
\end{proof}

Next example shows that the $\RR$-rationality condition in Proposition \ref{prop: dP3 Z3-minim} cannot be omitted.

\begin{ex}\label{ex: counterexample}
	Let $S_\alpha$ be the cubic surface in $\PP_\RR^3$ given by the equation
	\[
	\alpha x_0^3+x_1^3+x_2^3+x_3^3-(x_0+x_1+x_2+x_3)^3=0
	\]
	It can be shown that for $1/16<\alpha<1/4$ the set of real points $S_\alpha(\RR)$ is not connected and homeomorphic to $\SSS^2\sqcup\RR\PP^2$ (one can find a detailed study of the topology of $S_\alpha(\RR)$ and lines on $S_\alpha$ in \cite{polotop}). In particular, $S_\alpha$ are not $\RR$-rational for such $\alpha$'s. There are only 3 real lines on $S_\alpha$ which are given by the equations
	\[
	\ell_1:\ x_0=x_1+x_2=0,\ \ \ \ell_2:\ x_0=x_2+x_3=0,\ \ \ \ell_3:\ x_0=x_1+x_3=0.
	\]
	These lines form a triangle: \[\ell_1\cap\ell_2=[0:1:-1:1],\ \ \ell_1\cap\ell_3=[0:-1:1:1],\ \ \ell_2\cap\ell_3=[0:-1:-1:1].\]
	The cyclic permutation of the coordinates $g:\ x_1\mapsto x_2\mapsto x_3$ induces the permutation of lines: $\ell_1\mapsto\ell_2\mapsto\ell_3$. Besides, there are no disjoint complex conjugate lines on $S_\alpha$. Indeed, otherwise $S_\alpha$ is birationally trivial over $\RR$ (via the map (\ref{eq: bir trivial cubic}) in Proposition \ref{prop: dp5 order 5}) and $S_\alpha(\RR)$ must be connected. So, $S_\alpha$ is $g$-minimal.
\end{ex} 

\subsection{del Pezzo surfaces of degree $2$}\label{sec:dp2}
In this subsection $X$ denotes a real del Pezzo surface of degree 2. Recall that the anticanonical map 
\[
\varphi_{|-K_X|}: X\rightarrow\PR
\]
is a double cover branched over a smooth quartic $B\subset\PR$. The Galois involution $\gamma$ of the double cover is called the {\it Geiser involution}. Let $F(x,y,z)=0$ be the equation of $B$. Then $X$ can be given by the equation
\[
w^2=F(x,y,z)
\]
in the weighted projective space $\mathbb{P}(1,1,1,2)$.
\begin{rem}\label{geiserminimality}
Recall that we denoted by $\mathbb{E}_7$ the sublattice in $\Pic(X_\CC)$ generated by the root system $\E_7$. It is known that the Geiser involution $\gamma$ acts as the minus identity in $\mathbb{E}_7$ \cite[6.6]{di}. Moreover, $\rk\Pic(X_\CC)^\gamma=1$, so a del Pezzo surface $X_\CC$ of degree 2 is always $\gamma$-minimal.
\end{rem}

It is clear that $B$ should be invariant with respect to any automorphism of $X$, so there exists a homomorphism 
\[
\chi: \Aut(X)\rightarrow\Aut(B),
\] 
whose kernel is $\langle\gamma\rangle$. In fact, one can easily see that $\Aut(B)\cong\Aut(X)/\langle\gamma\rangle$. As $G$ has odd order, we have $G\subset\Aut(B)\subset\PGL_3(\RR)$, so $G$ is cyclic by Proposition \ref{PGLclassification}. 

Denote by $g$ a generator of $G$ whose order equals $n$. Choose coordinates in such a way that the action of $g$ on $H^0(X,-K_X)\otimes_\RR\CC\cong\CC^3$ has the form
\[
(x,y,z)\mapsto (x,\omega_n^ky,\omega_n^{-k}z),\ \ 0<k<n.
\]
If $n\geq 5$, then $F(x,y,z)$ is a linear combination of the monomials $x^4$, $x^2yz$ and $y^2z^2$, so $B$ is singular at the point $[0:1:0]$. Therefore, it remains to consider the case
\[
G=\langle g\rangle_3\cong\ZZ/3\ZZ.
\]
Denote by $B'$ the quotient curve $B/G$. Then, by Riemann–Hurwitz formula, we have
\[
2-2g(B)=|G|\left ( 2-2g(B')-\sum_{x\in B}\left ( 1-\frac{1}{|\stab x|} \right )\right ),
\]
where $\stab x$ denotes the stabilizer subgroup of a point $x\in B$. Let $N$ be the number of points on $B$ fixed by $G$. Since $g(B)=3$ and $G\cong\ZZ/3\ZZ$, we have
\[
N=5-3g(B'),
\]
so either $N=2$, or $N=5$. Obviously, an element $g\in\PGL_3(\RR)$ of order 3 cannot have  five (possibly nonreal) fixed points, so it remains to consider the first case $N=2$.

Note that there is the third fixed point $p\notin B(\CC)$ (which is real). It means that we have 4 fixed points on $\XC$ in total.

Recall that there is a homomorphism
\[
\eta\times\rho: \Gamma\times G\rightarrow\WW(\E_7).
\]
\begin{lem}\label{delPezzo2Lemma}
Let $X$ be a real $G$-minimal del Pezzo surface of degree 2 with $X(\RR)\ne\varnothing$, where the order of $G$ is odd. Then $\eta(\Gamma)\ne{\rm id}.$
\end{lem}
\begin{proof}
Assuming that $\eta(\Gamma)={\rm id}$, we get ${\rm rk}\Pic(\XC)^G=1$. Let $E_1,\ E_2,\ldots, E_s$ be $s$ $(-1)$-curves on $\XC$, forming an orbit of $G$. Then
$
E_1+\ldots+E_s=aK_{\XC}$, $a\in\ZZ$,
so
\[
2a=aK_{\XC}^2=\sum_{i=1}^{s}(E_i\cdot K_{\XC})=\sum_{i=1}^{s}(-1)=-s.
\]
It follows that $s$ is even, hence $|G|$ is even too, a contradiction.
\end{proof}

Lemma \ref{delPezzo2Lemma} shows that the complex conjugation $\sigma\in\Gamma$ gives a nontrivial element $\sigma^*\in\WW(\E_7)$. It means that $(g\circ\sigma)^*$ is an element of order 6 in $\WW(\E_7)$. All 17 classes of elements of order 6 in $\WW(\E_7)$ are listed in Table \ref{table:appendixE7} (see Appendix \ref{appendix}). Since $1\notin\Sp(g\circ\sigma)^*$ by Lemma \ref{cyclicminimality}, there are in fact only four possibilities for $(g\circ\sigma)^*$:
 
\begin{table}[H]
\caption{Possibilities for $(g\circ\sigma)^*$}
\label{table:PossibleCasesDP2}
\begin{tabular}{cccclll}
\cline{1-3}
\multicolumn{1}{|c}{\text{Carter graph}} & \multicolumn{1}{|c}{\text{Characteristic polynomial}} & \multicolumn{1}{|c|}{$\tr(g\circ\sigma)^*$} &  &  \\ \cline{1-3}

\multicolumn{1}{|c}{$A_5\times A_2$} & \multicolumn{1}{|c}{$(t^5+t^4+t^3+t^2+t+1)(t^2+t+1)$} & \multicolumn{1}{|c|}{$-2$} &  &  \\ \cline{1-3}
\multicolumn{1}{|c}{$D_4\times A_1^3$} & \multicolumn{1}{|c}{$(t^3+1)(t+1)^4$} & \multicolumn{1}{|c|}{$-4$} &  &  \\ \cline{1-3}
\multicolumn{1}{|c}{$D_6(a_2)\times A_1$} & \multicolumn{1}{|c}{$(t^3+1)^2(t+1)$} & \multicolumn{1}{|c|}{$-1$} &  &  \\ \cline{1-3}
\multicolumn{1}{|c}{$E_7(a_4)$} & \multicolumn{1}{|c}{$(t^2-t+1)^2(t^3+1)$} & \multicolumn{1}{|c|}{2} &  &  \\ \cline{1-3}

\end{tabular}
\end{table}

Since $g$ has exactly 4 fixed points on $\XC$, we have $\tr g^*=1$ by the Lefschetz fixed point formula (\ref{LefschetzFixedPointFormula}). According to Table \ref{table:appendixE7}, such $g^*$ belongs to the class $A_2^2$ and
\[\Sp(g^*)=\{1,1,1,\omega_3,\overline{\omega}_3,\omega_3,\overline{\omega}_3\}.\] 

As $X_\CC$ is $g\circ\sigma$-minimal, we have the following possibilities for $\Sp(g\circ\sigma)^*$ by Lemma \ref{cyclicminimality}:
\begin{table}[H]
\begin{tabular}{|c|c|c|ll}
\cline{1-3}
\text{Eigenvalues} & \text{Characteristic polynomial} & $\tr(g\circ\sigma)^*$ \\ \cline{1-3}
  $-1,-1,-1,\omega_3,\overline{\omega}_3,\omega_3,\overline{\omega}_3$ & $(t+1)^3(t^2+t+1)^2$ & $-5$  \\ \cline{1-3}
  $-1,-1,-1,\omega_3,\overline{\omega}_3,-\omega_3,-\overline{\omega}_3$ & $(t+1)^3(t^2+t+1)(t^2-t+1)$ & $-3$  \\ \cline{1-3}
  $-1,-1,-1,-\omega_3,-\overline{\omega}_3,-\omega_3,-\overline{\omega}_3$ & $(t+1)^3(t^2-t+1)^2$ & $-1$  \\ \cline{1-3}
\end{tabular}
\end{table}
Comparing it with the data of Table \ref{table:PossibleCasesDP2}, we see that $(g\circ\sigma)^*$ belongs to the class $D_6(a_2)\times A_1$. The complex conjugation $\sigma$ acts on $K_{X_\CC}^\bot$ as minus identity, so it coincides with the Geiser involution $\gamma$. It follows from Remark \ref{geiserminimality} that $X$ is $\RR$-minimal. Therefore, $X$ is not $\RR$-rational by Theorem \ref{IskovskikhCrit}.

\begin{rem}\label{rem: dP2 non rational}
	As in the case $K_X^2=3$, the $\RR$-rationality condition cannot be omitted. Namely, there exists a real del Pezzo surface $X$ of degree 2 with a minimal action of $\ZZ/3\ZZ$ such that $X$ is not $\RR$-rational. To construct such a surface, consider seven different real points $[1:0:0]$, $[0:1:0]$, $[0:0:1]$, $[a:b:c]$, $[b:c:a]$, $[c:a:b]$, $[1:1:1]$ on $\PP_\RR^2$ in general position. Let $g\in\Aut(\PP_\RR^2)$ be a cyclic permutation of coordinates. Then our set of points is $g$-invariant, and their blow-up is a $\RR$-rational del Pezzo  surface of degree 2 with an action of $G=\ZZ/3\ZZ$. Note that all $(-1)$-curves on $X$ are real. Assume that $X$ is given by $w^2=F(x,y,z)$ in $\PP(1,1,1,2)$. Consider the surface $X'$, given by $w^2=-F(x,y,z)$. Then $X'$ is a real del Pezzo surface of degree 2 with $\sigma^*=\gamma^*$ in $\WW(\E_7)$. It is minimal over $\RR$ (in particular, $G$-minimal), hence not rational over $\RR$ by Theorem \ref{rminimal}. It is known that $X'(\RR)\approx\sqcup 4\SSS^2$, and $B(\RR)$ consists of four ovals, see \cite{kol} or \cite{Wall}.
\end{rem}

\subsection{del Pezzo surfaces of degree $1$}\label{sec:dp1}

In this subsection $X$ denotes a real del Pezzo surface of degree 1. The linear system $|-K_X|$ has a single base point $q$ and determines a rational map $\varphi: X\dashrightarrow S=\PRone$. Blowing $q$ up, we get the following commutative diagram:
\[
\xymatrix{
& \widetilde{X}\ar[ld]_{\pi}\ar[rd]^{\widetilde{\varphi}} & \\
X \ar@{-->}[rr]_{\varphi} & & S }
\]
where $\widetilde{\varphi}$ is an elliptic pencil. The linear system $|-2K_X|$ has no base points and exhibits $X$ as a double cover of a quadratic cone $Q\subset\mathbb{P}_\RR^3$ ramified over the vertex of $Q$ and a smooth curve $Q\cap Y$, where $Y$ is a cubic surface. The corresponding Galois involution $\beta$ is called the {\it Bertini involution}.
\begin{rem}
One can show that the Bertini involution $\beta$ acts as the minus identity in $\mathbb{E}_8$ and a del Pezzo surface $X_\CC$ of degree 1 is always $\beta$-minimal.
\end{rem}
Note that $q$ must be real and it is a fixed point for any automorphism group $G\subset\Aut(\XC)$. It follows that there is the natural faithful representation
\[
G\rightarrow {\rm GL}(T_qX)\cong{\rm GL_2}(\mathbb{R}),
\]
so $G$ is a cyclic group of odd order. The tables of conjugacy classes in $\WW(\E_8)$ show that the order of $G$ can be equal to $3,\ 5,\ 7,\ 9$ or $15$ \cite[Table 11]{weyl}. 

Every singular member of the linear system $|-K_{\XC}|$ is an irreducible curve of arithmetic genus 1. Therefore, it is a rational curve with a unique singularity which is either a node or a simple cusp. Denote by $n_{\rm cusp}$ the number of cuspidal curves $Z_{\rm cusp}$ and by $n_{\rm node}$ the number of nodal curves $Z_{\rm node}$ in $|-K_{\XC}|$.
\begin{lem}\label{topolemma}
We have
\[
n_{\rm node}+2n_{\rm cusp}=12.
\]
\end{lem}
\begin{proof}
All that we need is to compute the topological Euler characteristic of $\widetilde{X}_\CC$. Namely,
\[
\Eu(\widetilde{X}_\CC)=n_{\rm node}\Eu(Z_{\rm node})+n_{\rm cusp}\Eu(Z_{\rm cusp})=n_{\rm node}+2n_{\rm cusp}.
\] 
On the other hand,
\[
\Eu(\widetilde{X}_\CC)=\Eu(\XC)+1=\Eu(\PC)+8+1=12.
\]
\end{proof}

The action of $G$ on the pencil $|-K_X|$ induces the action on $S=\PRone$. This gives us the natural homomorphism $\mu: G\to\Aut(S)={\rm PGL}_2(\RR)$. Consider two cases.
\newline
\newline
{\bf Case} $\mu(G)={\rm id}$. Since $S$ can be naturally identified with $\PP(T_qX)$, the image of $G$ in $\GL(T_qX)$ consists of scalar matrices. Obviously, this is impossible because the order of $G$ is odd.
\newline
\newline
{\bf Case} $\mu(G)\ne {\rm id}$. There are exactly two conjugate imaginary fixed points on $S_\CC\cong\PCone$ (since the order of $\mu(G)$ is not 2). These points correspond to complex conjugate $G$-invariant curves $C$ and $\sigma(C)=\conj{C}$ in the linear system $|-K_{\XC}|$. We have three different cases.

{\bf a)} Let $C$ and $\conj{C}$ be nodal curves. Consider the normalization $\nu: \widehat{C}\to C$. Then the cyclic group $G$ has three fixed points $\nu^{-1}(\textrm{node})$ and $\nu^{-1}(q)$ on $\widehat{C}\cong\PCone$. Hence, $G$ acts trivially on $C$, a contradiction. 

{\bf b)} Now let $C$ and $\conj{C}$ be cuspidal curves. Put $n_{\rm cusp}=n_{\rm cusp}'+2$. Then $n_{\rm node}+2n_{\rm cusp}'=8$, so we have the following possibilities for a pair $(n_{\rm node},n_{\rm cusp}')$:
\[
(0,4),\ (2,3), (4,2),\ (6,1),\ (8,0).
\]
It is obvious that none of these cases occurs, as the curves of the same singularity type must be exchanged by $G$.

{\bf c)} Finally, let $C$ and $\conj{C}$ be smooth elliptic curves. It is well-known that the order of any automorphism of an elliptic curve, preserving the group structure, divides 24. Thus $G=\langle g\rangle_3\cong\ZZ/3\ZZ$. There are exactly 3 fixed points on each curve and $\{q\}=C(\RR)=\conj{C}(\RR)$ is the only {\it real} point fixed by $G$. Note that we have 5 fixed points in total.  By the Lefschetz fixed point formula,
\[
\tr(g^*)=\#\Fix_{\XC}(g)-3=2.
\]
To find a specific type of action, we turn to the tables of conjugacy classes in $\WW(\E_8)$. Now we are interested in elements of order 3.
\begin{table}[H]
\caption{Elements of order 3 in $\WW(\E_8)$}
\label{table: elements3weyl8}
\begin{tabular}{|c|c|c|ll}
\cline{1-3}
\text{Carter graph} & \text{Characteristic polynomial} & Trace on $K_{X_\CC}^\bot$ \\ \cline{1-3}
  $A_2$ & $(t^2+t+1)(t-1)^6$ & 5  \\ \cline{1-3}
  $A_2^2$ & $(t^2+t+1)^2(t-1)^4$ & 2  \\ \cline{1-3}
  $A_2^3$ & $(t^2+t+1)^3(t-1)^2$ & $-1$  \\ \cline{1-3}
  $A_2^4$ & $(t^2+t+1)^4$ & $-4$  \\ \cline{1-3}
\end{tabular}
\end{table}
We see that $g^*$ belongs to the class $A_2^2$ and \[\Sp(g^*)=\{1,1,1,1,\omega_3,\overline{\omega}_3,\omega_3,\overline{\omega}_3\}.\] According to Lemma \ref{cyclicminimality}, a surface $\XC$ is not $\langle g\rangle$-minimal for such $g$. Thus $\eta(\Gamma)\ne{\rm id}$ and we are looking for elements of order 6 in $\WW(\E_8)$. Note that there are only 3 possibilities for $\Sp(g\circ\sigma)^*$:

\begin{table}[H]
\begin{tabular}{|c|c|c|ll}
\cline{1-3}
\text{Eigenvalues} & \text{Characteristic polynomial} & $\tr(g\circ\sigma)^*$ \\ \cline{1-3}
  $-1,-1,-1,-1,\omega_3,\overline{\omega}_3,\omega_3,\overline{\omega}_3$ & $(t+1)^4(t^2+t+1)^2$ & $-6$  \\ \cline{1-3}
  $-1,-1,-1,-1,-\omega_3,-\overline{\omega}_3,\omega_3,\overline{\omega}_3$ & $(t+1)^4(t^2+t+1)(t^2-t+1)$ & $-4$  \\ \cline{1-3}
  $-1,-1,-1,-1,-\omega_3,-\overline{\omega}_3,-\omega_3,-\overline{\omega}_3$ & $(t+1)^4(t^2-t+1)^2$ & $-2$  \\ \cline{1-3}
\end{tabular}
\end{table}

In Table \ref{table:appendixE8} (see Appendix \ref{appendix}) we list the conjugacy classes of elements of order 6 in $\WW(\E_8)$. It turns out that only the third case in the table above actually occurs. Such an element belongs to the class $D_4^2$. Moreover, we get that the complex involution acts on $K_{\XC}^\bot$ as minus identity, i.e. coincides with the Bertini involution. It follows that $X$ is $\RR$-minimal. Finally, according to Theorem \ref{IskovskikhCrit}, $X$ fails to be rational over $\RR$.

We close this section by proving that, unlike the cases $K_X^2=2$ and $K_X^2=3$, we do not really need $X$ to be rational over $\RR$, and may only assume $X(\RR)\ne\varnothing$.
\begin{prop}
Let $X$ be a real del Pezzo surface of degree 1 with $X(\RR)\ne\varnothing$ and $G\subset\Aut(X)$ is a subgroup of odd order. Then $X$ is not $G$-minimal.
\end{prop}
\begin{proof}
Clearly, it is sufficient to prove that $G=\ZZ/3\ZZ$ cannot act minimally on $X$. Assume the contrary. As it was shown above, there is a single real fixed point $q\in X$ (the base point of the elliptic pencil). Moreover, $X$ has to be minimal over $\RR$. According to \cite[Theorem 6.8]{kol}, we have
\[
X(\RR)\approx\mathbb{RP}^2\sqcup4\SSS^2.
\]
Obviously, at least one sphere must be $G$-invariant. On the other hand, any continuous map of odd order from $\SSS^2$ to itself has a fixed point and the same is true for any continuous map $\mathbb{RP}^2\to \mathbb{RP}^2$ (see e.g. \cite[Chapter 2, \S 2C]{Hat}). Therefore, there are at least two real fixed points, a contradiction.
\end{proof}

\newpage

\appendix

\section{Conjugacy classes in some Weyl groups}\label{appendix}

{\bf Notation}. We denote by $p_k$ a polynomial of the form $t^k+t^{k-1}+\ldots+t+1$. 

\begin{table}[H]
\caption{Elements of order 2, 3 and 6 in $\WW(\E_7)$}
\label{table:appendixE7}
\begin{tabular}{cccll}
\cline{1-3}

\multicolumn{1}{|c}{Order} & \multicolumn{1}{|c}{\text{Carter graph}} & \multicolumn{1}{|c|}{\text{Characteristic polynomial}} &  &  \\ \cline{1-3}
\multicolumn{1}{|c}{2} & \multicolumn{1}{|c}{$A_1$} & \multicolumn{1}{|c|}{$p_1(t-1)^6$}  &  &  \\ \cline{1-3}
\multicolumn{1}{|c}{2} & \multicolumn{1}{|c}{$A_1^2$} & \multicolumn{1}{|c|}{$p_1^2(t-1)^5$}  &  &  \\ \cline{1-3}
\multicolumn{1}{|c}{2} & \multicolumn{1}{|c}{$(A_1^3)'$} & \multicolumn{1}{|c|}{$p_1^3(t-1)^4$}  &  &  \\ \cline{1-3}
\multicolumn{1}{|c}{2} & \multicolumn{1}{|c}{$(A_1^3)''$} & \multicolumn{1}{|c|}{$p_1^3(t-1)^4$}  &  &  \\ \cline{1-3}
\multicolumn{1}{|c}{2} & \multicolumn{1}{|c}{$(A_1^4)'$} & \multicolumn{1}{|c|}{$p_1^4(t-1)^3$}  &  &  \\ \cline{1-3}
\multicolumn{1}{|c}{2} & \multicolumn{1}{|c}{$(A_1^4)''$} & \multicolumn{1}{|c|}{$p_1^4(t-1)^3$} &  &  \\ \cline{1-3}
\multicolumn{1}{|c}{2} & \multicolumn{1}{|c}{$A_1^5$} & \multicolumn{1}{|c|}{$p_1^5(t-1)^2$}  &  &  \\ \cline{1-3}
\multicolumn{1}{|c}{2} & \multicolumn{1}{|c}{$A_1^6$} & \multicolumn{1}{|c|}{$p_1^6(t-1)$}  &  &  \\ \cline{1-3}
\multicolumn{1}{|c}{2} & \multicolumn{1}{|c}{$A_1^7$} & \multicolumn{1}{|c|}{$p_1^7$}  &  &  \\ \cline{1-3}

\multicolumn{1}{|c}{3} & \multicolumn{1}{|c}{$A_2$} & \multicolumn{1}{|c|}{$p_2(t-1)^5$}  &  &  \\ \cline{1-3}
\multicolumn{1}{|c}{3} & \multicolumn{1}{|c}{$A_2^2$} & \multicolumn{1}{|c|}{$p_2^2(t-1)^3$}  &  &  \\ \cline{1-3}
\multicolumn{1}{|c}{3} & \multicolumn{1}{|c}{$A_2^3$} & \multicolumn{1}{|c|}{$p_2^3(t-1)$}  &  &  \\ \cline{1-3}

\multicolumn{1}{|c}{6} & \multicolumn{1}{|c}{$A_2\times A_1$} & \multicolumn{1}{|c|}{$p_2p_1(t-1)^4$} &  &  \\ \cline{1-3}
\multicolumn{1}{|c}{6} & \multicolumn{1}{|c}{$A_2\times A_1^2$} & \multicolumn{1}{|c|}{$p_2p_1^2(t-1)^3$}  &  &  \\ \cline{1-3}
\multicolumn{1}{|c}{6} & \multicolumn{1}{|c}{$D_4$} & \multicolumn{1}{|c|}{$(t^3+1)(t+1)(t-1)^3$}  &  &  \\ \cline{1-3}
\multicolumn{1}{|c}{6} & \multicolumn{1}{|c}{$A_2\times A_1^3$} & \multicolumn{1}{|c|}{$p_2p_1^3(t-1)^2$}  &  &  \\ \cline{1-3}
\multicolumn{1}{|c}{6} & \multicolumn{1}{|c}{$A_2^2\times A_1$} & \multicolumn{1}{|c|}{$p_2^2p_1(t-1)^2$}  &  &  \\ \cline{1-3}
\multicolumn{1}{|c}{6} & \multicolumn{1}{|c}{$(A_5)'$} & \multicolumn{1}{|c|}{$p_5(t-1)^2$}  &  &  \\ \cline{1-3}
\multicolumn{1}{|c}{6} & \multicolumn{1}{|c}{$(A_5)''$} & \multicolumn{1}{|c|}{$p_5(t-1)^2$} &  &  \\ \cline{1-3}
\multicolumn{1}{|c}{6} & \multicolumn{1}{|c}{$D_4\times A_1$} & \multicolumn{1}{|c|}{$(t^3+1)(t+1)^2(t-1)^2$}  &  &  \\ \cline{1-3}
\multicolumn{1}{|c}{6} & \multicolumn{1}{|c}{$(A_5\times A_1)'$} & \multicolumn{1}{|c|}{$p_5p_1(t-1)$}  &  &  \\ \cline{1-3}
\multicolumn{1}{|c}{6} & \multicolumn{1}{|c}{$(A_5\times A_1)''$} & \multicolumn{1}{|c|}{$p_5p_1(t-1)$} &  &  \\ \cline{1-3}
\multicolumn{1}{|c}{6} & \multicolumn{1}{|c}{$D_4\times A_1^2$} & \multicolumn{1}{|c|}{$(t^3+1)(t+1)^3(t-1)$}  &  &  \\ \cline{1-3}
\multicolumn{1}{|c}{6} & \multicolumn{1}{|c}{$D_6(a_2)$} & \multicolumn{1}{|c|}{$(t^3+1)^2(t-1)$}  &  &  \\ \cline{1-3}
\multicolumn{1}{|c}{6} & \multicolumn{1}{|c}{$E_6(a_2)$} & \multicolumn{1}{|c|}{$(t^2+t+1)(t^2-t+1)^2(t-1)$}  &  &  \\ \cline{1-3}
\multicolumn{1}{|c}{6} & \multicolumn{1}{|c}{$A_5\times A_2$} & \multicolumn{1}{|c|}{$p_5p_2$}  &  &  \\ \cline{1-3}
\multicolumn{1}{|c}{6} & \multicolumn{1}{|c}{$D_4\times A_1^3$} & \multicolumn{1}{|c|}{$(t^3+1)(t+1)^4$}  &  &  \\ \cline{1-3}
\multicolumn{1}{|c}{6} & \multicolumn{1}{|c}{$D_6(a_2)\times A_1$} & \multicolumn{1}{|c|}{$(t^3+1)^2(t+1)$} &  &  \\ \cline{1-3}
\multicolumn{1}{|c}{6} & \multicolumn{1}{|c}{$E_7(a_4)$} & \multicolumn{1}{|c|}{$(t^2-t+1)^2(t^3+1)$}  &  &  \\ \cline{1-3}

\end{tabular}
\end{table}

\newpage

\setlength{\extrarowheight}{3pt}
\begin{table}[H]
\caption{Elements of order 6 in $\WW(\E_8)$}
\label{table:appendixE8}
\begin{tabular}{cccll}
\cline{1-3}
\multicolumn{1}{|c}{Order} & \multicolumn{1}{|c}{Carter graph} & \multicolumn{1}{|c|}{Characteristic polynomial} &  &  \\ \cline{1-3}

\multicolumn{1}{|c}{6} & \multicolumn{1}{|c}{$A_2\times A_1$} & \multicolumn{1}{|c|}{$p_2p_1(t-1)^5$}  &  &  \\ \cline{1-3}
\multicolumn{1}{|c}{6} & \multicolumn{1}{|c}{$A_2\times A_1^2$} & \multicolumn{1}{|c|}{$p_2p_1^2(t-1)^4$}  &  &  \\ \cline{1-3}
\multicolumn{1}{|c}{6} & \multicolumn{1}{|c}{$D_4$} & \multicolumn{1}{|c|}{$(t^3+1)(t+1)(t-1)^4$} &  &  \\ \cline{1-3}
\multicolumn{1}{|c}{6} & \multicolumn{1}{|c}{$A_2\times A_1^3$} & \multicolumn{1}{|c|}{$p_2p_1^3(t-1)^3$}  &  &  \\ \cline{1-3}
\multicolumn{1}{|c}{6} & \multicolumn{1}{|c}{$A_2^2\times A_1$} & \multicolumn{1}{|c|}{$p_2^2p_1(t-1)^3$}  &  &  \\ \cline{1-3}
\multicolumn{1}{|c}{6} & \multicolumn{1}{|c}{$A_5$} & \multicolumn{1}{|c|}{$p_5(t-1)^3$}  &  &  \\ \cline{1-3}
\multicolumn{1}{|c}{6} & \multicolumn{1}{|c}{$D_4\times A_1$} & \multicolumn{1}{|c|}{$(t^3+1)(t+1)^2(t-1)^3$}  &  &  \\ \cline{1-3}
\multicolumn{1}{|c}{6} & \multicolumn{1}{|c}{$A_2\times A_1^4$} & \multicolumn{1}{|c|}{$p_2p_1^4(t-1)^2$} &  &  \\ \cline{1-3}
\multicolumn{1}{|c}{6} & \multicolumn{1}{|c}{$A_2^2\times A_1^2$} & \multicolumn{1}{|c|}{$p_2^2p_1^2(t-1)^2$} &  &  \\ \cline{1-3}
\multicolumn{1}{|c}{6} & \multicolumn{1}{|c}{$(A_5\times A_1)'$} & \multicolumn{1}{|c|}{$p_5p_1(t-1)^2$} &  &  \\ \cline{1-3}
\multicolumn{1}{|c}{6} & \multicolumn{1}{|c}{$(A_5\times A_1)''$} & \multicolumn{1}{|c|}{$p_5p_1(t-1)^2$} &  &  \\ \cline{1-3}
\multicolumn{1}{|c}{6} & \multicolumn{1}{|c}{$D_4\times A_1^2$} & \multicolumn{1}{|c|}{$(t^3+1)(t+1)^3(t-1)^2$} &  &  \\ \cline{1-3}
\multicolumn{1}{|c}{6} & \multicolumn{1}{|c}{$D_4\times A_2$} & \multicolumn{1}{|c|}{$p_2(t^3+1)(t+1)(t-1)^2$} &  &  \\ \cline{1-3}
\multicolumn{1}{|c}{6} & \multicolumn{1}{|c}{$D_6(a_2)$} & \multicolumn{1}{|c|}{$(t^3+1)^2(t-1)^2$}  &  &  \\ \cline{1-3}
\multicolumn{1}{|c}{6} & \multicolumn{1}{|c}{$E_6(a_2)$} & \multicolumn{1}{|c|}{$(t^2+t+1)(t^2-t+1)^2(t-1)^2$}  &  &  \\ \cline{1-3}
\multicolumn{1}{|c}{6} & \multicolumn{1}{|c}{$A_2^3\times A_1$} & \multicolumn{1}{|c|}{$p_2^3p_1(t-1)$}  &  &  \\ \cline{1-3}
\multicolumn{1}{|c}{6} & \multicolumn{1}{|c}{$A_5\times A_1^2$} & \multicolumn{1}{|c|}{$p_5p_1^2(t-1)$}  &  &  \\ \cline{1-3}
\multicolumn{1}{|c}{6} & \multicolumn{1}{|c}{$A_5\times A_2$} & \multicolumn{1}{|c|}{$p_5p_2(t-1)$}  &  &  \\ \cline{1-3}
\multicolumn{1}{|c}{6} & \multicolumn{1}{|c}{$D_4\times A_1^3$} & \multicolumn{1}{|c|}{$(t^3+1)(t+1)^4(t-1)$}  &  &  \\ \cline{1-3}
\multicolumn{1}{|c}{6} & \multicolumn{1}{|c}{$D_6(a_2)\times A_1$} & \multicolumn{1}{|c|}{$(t^3+1)^2(t+1)(t-1)$}  &  &  \\ \cline{1-3}
\multicolumn{1}{|c}{6} & \multicolumn{1}{|c}{$E_6(a_2)\times A_1$} & \multicolumn{1}{|c|}{$(t^2-t+1)^2(t^2+t+1)(t+1)(t-1)$}  &  &  \\ \cline{1-3}
\multicolumn{1}{|c}{6} & \multicolumn{1}{|c}{$E_7(a_4)$} & \multicolumn{1}{|c|}{$(t^2-t+1)^2(t^3+1)(t-1)$}  &  &  \\ \cline{1-3}
\multicolumn{1}{|c}{6} & \multicolumn{1}{|c}{$A_5\times A_2\times A_1$} & \multicolumn{1}{|c|}{$p_5p_2p_1$}  &  &  \\ \cline{1-3}
\multicolumn{1}{|c}{6} & \multicolumn{1}{|c}{$D_4\times A_1^4$} & \multicolumn{1}{|c|}{$(t^3+1)(t+1)^5$}  &  &  \\ \cline{1-3}
\multicolumn{1}{|c}{6} & \multicolumn{1}{|c}{$D_4^2$} & \multicolumn{1}{|c|}{$(t^3+1)^2(t+1)^2$}  &  &  \\ \cline{1-3}
\multicolumn{1}{|c}{6} & \multicolumn{1}{|c}{$E_6(a_2)\times A_2$} & \multicolumn{1}{|c|}{$p_2(t^2-t+1)^2(t^2+t+1)$}  &  &  \\ \cline{1-3}
\multicolumn{1}{|c}{6} & \multicolumn{1}{|c}{$E_7(a_4)\times A_1$} & \multicolumn{1}{|c|}{$p_1(t^2-t+1)^2(t^3+1)$}  &  &  \\ \cline{1-3}
\multicolumn{1}{|c}{6} & \multicolumn{1}{|c}{$E_8(a_8)$} & \multicolumn{1}{|c|}{$(t^2-t+1)^4$}  &  &  \\ \cline{1-3}

\end{tabular}
\end{table}

\newpage

\def\bibindent{2.5em}


\begin{thebibliography}{99\kern\bibindent}

\bibitem[BaBe00]{bb} L. Bayle, A. Beauville, {\it Birational involutions of $\mathbb{P}^2$}, Asian J. Math. 4 (2000), no.1, 11--17.
\bibitem[BeBl04]{bovblanc} A. Beauville, J. Blanc, {\it On Cremona transformations of prime order}, C.R. Acad. Sci. Paris Ser. I 339 (2004), no. 4, 257--259.
\bibitem[Bla09]{blancabelian} J. Blanc, {\it Linearisation of finite Abelian subgroups of the Cremona group of the plane}, Groups Geom. Dyn. 3 (2009), no. 2, 215--266.
\bibitem[BlMa13]{blanc} J. Blanc and F. Mangolte, {\it Cremona groups of real surfaces}, Proceedings of GABAG2012, (2013).
\bibitem[Car72]{weyl} R.W. Carter, {\it Conjugacy classes in the Weyl group}, Composito Mathematica, vol. 25 (1972), 1--59.
\bibitem[Com12]{comm} A. Comessatti, {\it Fondamenti per la geometria sopra superfizie razionali dal punto di vista reale}, Math. Ann. 73 (1912), 1--72.
\bibitem[Cor95]{corti} A. Corti, {\it Factoring birational maps of threefolds after Sarkisov}, Journal of Algebraic Geometry, vol. 4 (1995), 223--254.
\bibitem[dFe04]{df} T. de Fernex, {\it On planar Cremona maps of prime order}, Nagoya Math J., vol. 174 (2004), 1--28.
\bibitem[DI09a]{di} I.V. Dolgachev, V.A. Iskovskikh, {\it Finite subgroups of the plane Cremona group}, Algebra, arithmetic, and geometry: in honor of Yu. I. Manin, Progr. Math., vol. 269 (2009), Birkhauser Boston, Inc., Boston, MA., 443--558.
\bibitem[DI09b]{di-perf} I.V. Dolgachev, V.A. Iskovskikh, {\it On elements of prime order in the plane Cremona group over a perfect field}, Int. Math. Res. Notices (2009), no. 18, 3467--3485.
\bibitem[Dol12]{cag} I. V. Dolgachev, {\it Classical Algebraic Geometry: A Modern View}, Cambridge University Press, 1st edition, (2012).
\bibitem[Hat02]{Hat} A. Hatcher, {\it  Algebraic Topology}, Cambridge University Press, Cambridge (2002).
\bibitem[HM09]{huis-mang} J. Huisman, F. Mangolte, {\it The group of automorphisms of a real rational surface is n-transitive}, Bull. Lond. Math. Soc. 41.3 (2009), pp. 563--568.
\bibitem[Isk79]{isk-2} V. A. Iskovskikh, {\it Minimal models of rational surfaces over arbitrary fields}, Izv. Akad. Nauk SSSR Ser. Mat., 43:1 (1979), 19--43.
\bibitem[Isk96]{isk-1} V. A. Iskovskikh, {\it Factorization of birational maps of rational surfaces from the viewpoint of Mori theory}, Uspekhi Mat. Nauk, 51:4(310) (1996), 3--72.
\bibitem[Kol97]{kol} J. Koll\'{a}r, {\it Real Algebraic Surfaces}, Notes of the 1997 Trento summer school lectures, (preprint).
\bibitem[KM09]{kol-mang} J. Koll\'{a}r, F. Mangolte, {\it Cremona transformations and
diffeomorphisms of surfaces}, Adv. Math. 222.1 (2009), pp. 44--61.
\bibitem[Man67]{manin} Yu. I. Manin, {\it Rational surfaces over perfect fields. II}, Mat. Sb., 72(114):2 (1967), 161--192.
\bibitem[Man86]{cubicforms} Yu. I. Manin, {\it Cubic forms: Algebra, geometry, arithmetic}, North-Holland Mathematical Library, 4, Ed. 2, North-Holland Publishing Co., Amsterdam, (1986).
\bibitem[Old37]{trilinear} R. Oldenburger, {\it Real canonical binary trilinear forms}, American Journal of Mathematics, vol. 59, no. 2 (1937), 427--435.
\bibitem[PT08]{polotop} I. Polo-Blanco, J. Top, {\it Explicit Real Cubic Surfaces}, Canad. Math. Bull., vol. 51 (1), (2008), pp. 125--133
\bibitem[Pol97]{pol} Yu. M. Polyakova, {\it Factoring birational maps of rational surfaces over the field of real numbers}, Fundam. Prikl. Mat. 3:2 (1997), 519--547.
\bibitem[Pro12]{pro1} Yu. G. Prokhorov, {\it Simple finite subgroups of the Cremona group of rank 3}, J. Algebraic Geom. 21 (2012), 563--600.
\bibitem[Pro15]{pro2} Yu. G. Prokhorov, {\it On G-Fano threefolds}, Izv. RAN. Ser. Mat., 79:4 (2015), 159--174
\bibitem[Rob15]{rob} M. F. Robayo, {\it Prime order birational diffeomorphisms of the sphere}, Annali della Scuola normale superiore di Pisa, Classe di scienze, 2015, S. to appear. 
\bibitem[Ser08]{serre2} J.-P. Serre, {\it Le groupe de Cremona et ses sous-groupes finis}, Seminaire Bourbaki, no. 1000 (2008), 75--100.
\bibitem[Ser09]{serre1} J.-P. Serre, {\it A minkowski-style bound for the orders of the finite subgroups of the Cremona group of rank 2 over an arbitrary field}, Mosc. Math. J. 9:1 (2009), 183--198.
\bibitem[Silh89]{Silhol} R. Silhol, {\it Real Algebraic Surfaces}, Springer, (1989).
\bibitem[Vos98]{Voskr} V. E. Voskresenskii, {\it Algebraic Groups and Their Birational Invariants}, Transl. of Math. Monographs, 179, Amer. Math. Soc. (1998).
\bibitem[Wall87]{Wall} C. T. C. Wall, {\it Real forms of smooth del Pezzo surfaces}, J. Reine Angew. Math., 375/376 (1987), 47--66.

\end{thebibliography}
\end{document}